\documentclass[a4paper,11pt,pdftex]{article}

\usepackage{verbatim}
\usepackage{amsmath}
\usepackage{amssymb}
\usepackage{times}
\usepackage{amsthm}
\usepackage{xcolor}
\usepackage{enumerate}
\usepackage{enumitem}
\setlist[enumerate,1]{label={\normalfont(\roman*)}}
\usepackage{mathrsfs}
\usepackage[pdftex]{graphicx}
\usepackage{float}
\usepackage{wrapfig}
\usepackage{layout}
\usepackage{cite}
\usepackage{braket}
\usepackage[draft=false, setpagesize=false, pdfstartview=FitH,
colorlinks=true, linkcolor=blue, citecolor=magenta, pagebackref=true, bookmarks=false]
{hyperref}
 
\usepackage{lineno}
 \newcommand*\patchAmsMathEnvironmentForLineno[1]{
   \expandafter\let\csname old#1\expandafter\endcsname\csname #1\endcsname
   \expandafter\let\csname oldend#1\expandafter\endcsname\csname end#1\endcsname
   \renewenvironment{#1}
     {\linenomath\csname old#1\endcsname}
     {\csname oldend#1\endcsname\endlinenomath}}
 \newcommand*\patchBothAmsMathEnvironmentsForLineno[1]{
   \patchAmsMathEnvironmentForLineno{#1}
   \patchAmsMathEnvironmentForLineno{#1*}}
 \AtBeginDocument{
 \patchBothAmsMathEnvironmentsForLineno{equation}
 \patchBothAmsMathEnvironmentsForLineno{align}
 \patchBothAmsMathEnvironmentsForLineno{flalign}
 \patchBothAmsMathEnvironmentsForLineno{alignat}
 \patchBothAmsMathEnvironmentsForLineno{gather}
 \patchBothAmsMathEnvironmentsForLineno{multline}
 }

\theoremstyle{definition}
\newtheorem{definition}{Definition}[section]
\theoremstyle{plain}
\newtheorem{theorem}[definition]{Theorem}

\newtheorem{lemma}[definition]{Lemma}
\newtheorem{proposition}[definition]{Proposition}
\newtheorem{remark}[definition]{Remark}

\usepackage[truedimen,margin=15truemm, bottom=20truemm]{geometry}

\numberwithin{equation}{section}

\usepackage{authblk}
 
\title
	{\bf Uniqueness of bound states for sublinear elliptic equations }

  \author[a]{Chengxiang Zhang\thanks{  zcx@bnu.edu.cn}}
  \author[b]{Xu Zhang\thanks{ zhangxu0725@csu.edu.cn, darkblue1121@163.com, corresponding author  }}

  \affil[a]{\footnotesize Laboratory of Mathematics and Complex Systems (Ministry of Education), School of Mathematical Sciences, 
  
  Beijing Normal University, Beijing 100875, P. R. China}
  \affil[b]{\footnotesize  School of Mathematics and Statistics, HNP-LAMA, Central South University, Changsha 410083, P. R. China}
  
  \date{}
  
\begin{document}
  \maketitle
  \begin{minipage}{16.5cm}
{\small {\bf Abstract:}
	We investigate the uniqueness of radial bound state solutions to the sublinear elliptic equation
	\[
	\begin{cases}
		-\Delta u - u + |u|^{q-2}u = 0 & \text{in } \mathbb{R}^n, \\
		u(x) \to 0 & \text{as } |x| \to \infty,
	\end{cases}
	\]
	where $q\in(1,2)$ and $n\geq 2$. A distinctive feature of this problem is the non-Lipschitz singularity of the nonlinearity at the origin, which gives rise to compactly supported ground states and bound states. Using a shooting argument together with a detailed analysis of the linearized variation with respect to the initial value, we prove that for every prescribed integer $k\geq 1$, the equation admits exactly one radial bound state solution with $k$ simple zeros, up to sign reflection and spatial translation. In addition, our analysis yields a classification of radial solutions according to the initial value and describes their behavior near the finite support boundary.

    
    }
    
    \end{minipage}

\section{Introduction and main results}
For $n\geq 2$ and $q\in(1,2)$, we study the sublinear elliptic equation
\begin{equation}\label{eq:sublinear}
-\Delta u - u + |u|^{q-2}u = 0  \quad \text{in } \mathbb{R}^n.
\end{equation}
A solution $u$ is called a \emph{ground state} if $u(x)\to 0$ as $|x|\to\infty$, it does not change sign,
and the set $\{x \mid u(x)\neq 0\}$ is connected.
A \emph{bound state} is a solution that changes sign and also tends to $0$ at infinity.
Equation \eqref{eq:sublinear} appears naturally when studying the blow‑up behavior of the porous medium equation with a source term,
$u_t = \Delta u^m + u^m$ ($m>1$). 
It is well known that solutions to this parabolic problem blow up in finite time $T$.
Looking for self‑similar solutions of the form $u(x,t) = (T-t)^{-1/(m-1)} w(x)$ leads to the elliptic equation
$\Delta w^m + w^m - \frac{1}{m-1}w = 0$. Through the transformation $v = w^m$ (or a suitable rescaling)
this reduces to the sublinear equation \eqref{eq:sublinear} with $q = 1 + 1/m \in (1,2)$.
The blow-up set of the parabolic problem then corresponds exactly to the support of the solution $v$ (\cite{Cortazar1998,Cortazar2002,Gui1995}).

A distinctive feature of \eqref{eq:sublinear} is that all its ground states and bound states have \emph{compact support}.
This is a consequence of the non‑Lipschitz singularity of the nonlinearity $f(u)=u-|u|^{q-2}u$ at $u=0$;
any solution that reaches zero together with its derivative must vanish identically after that point, a phenomenon known as the compact support principle \cite{pucciStrongMaximumPrinciple2004}.

The qualitative properties of solutions to \eqref{eq:sublinear} have been studied extensively.
For the ground state, Gui~\cite{Gui1995} and Cort\'azar, Elgueta \& Felmer~\cite{Cortazar1996} proved that its support must be a ball
and the solution is radially symmetric. Their arguments overcame the lack of a classical Hopf boundary lemma by employing blow‑up techniques and refined moving plane methods.
Using these symmetry results, Cort\'azar, Del Pino \& Elgueta~\cite{Cortazar1998,Cortazar2002} completely characterized the regional blow‑up set of the associated porous medium equation, showing it consists of a finite union of disjoint balls.
More recently, Ikoma, Tanaka, Wang \& Zhang~\cite{Ikoma2022} gave a comprehensive analysis of the ground state. They  
established the precise asymptotic behavior of the support radius as $q\to1^+$ and $q\to2^-$, and studied convergence properties of the infinitely many nodal bound states.

While the existence and symmetry of the ground state are well understood, the uniqueness of \emph{nodal} bound states
(solutions with a prescribed number of sign changes) has remained open. The non-Lipschitz singularity at $u=0$ prevents a direct application of standard ODE uniqueness theorems and requires a careful study of the solution's behavior near its support radius.
The main purpose of this paper is to establish the uniqueness of radial bound states with a prescribed number of nodes for the sublinear equation \eqref{eq:sublinear}. We prove that for any integer $k\ge 1$, there exists exactly one radial bound state solution with exactly $k$ simple zeros.

The systematic investigation of semilinear elliptic equations of the form
\[
\Delta u + f(u)=0\quad\text{in }\mathbb{R}^n,\qquad u\in H^1(\mathbb{R}^n),\ u\not\equiv 0,
\]
was initiated in the celebrated works of Berestycki and Lions~\cite{BL1,BL2}. Under natural structural conditions on $f$, they proved the existence of infinitely many radial solutions: a positive, decreasing ground state $u_0$ and, for each integer $k\ge 1$, a sign‑changing bound state $u_k$ possessing exactly $k$ simple zeros. Their theory covers in particular the superlinear model
\[
f(u)=-u+|u|^{p-1}u,\qquad 1<p<\frac{n+2}{n-2}\;(n\ge 3),  
\]
which appears in the search for solitary waves of the nonlinear Klein--Gordon and Schr\"odinger equations~\cite{Tao}.
For this class of nonlinearities, Berestycki and Lions conjectured that for each $k$ there is exactly one radial bound state with $k$ nodes, up to translations and reflections.

The uniqueness of the ground state was established by Coffman~\cite{Coffman} for $n=3$, $p=3$, and later by Kwong~\cite{Kwong} and McLeod--Serrin~\cite{McLeodSerrin} for all admissible $p$ and $n\geq 2$; simplified proofs and generalizations were given by many authors (see~\cite{PS,ChenLin,Frank} and the references therein). The uniqueness of bound states, however, proved to be much more delicate. For the cubic equation in $\mathbb R^3$, the question was explicitly posed by Hastings and McLeod~\cite{HastingsMcLeod} as an open problem. Significant partial results were obtained by Cort\'azar, Garc\'ia‑Huidobro and Yarur~\cite{CGY} for $n=2,3,4$ with restricted exponents, by Ao, Wei and Yao~\cite{AoWeiYao} near the critical exponent, and by Cohen, Li and Schlag~\cite{CohenLiSchlag} who used computer‑assisted proofs to show uniqueness of the first twenty bound states in three dimensions.

A complete solution of the conjecture for the model case $f(u)=-u+|u|^{p-1}u$ in dimensions $n\ge 3$ was recently achieved by Tang~\cite{Tang2025}. His approach is based on a shooting argument for the radial initial value problem
\[
u''+\frac{n-1}{r}u'+f(u)=0,\qquad u(0)=\alpha>0,\; u'(0)=0,
\]
and a delicate analysis of the variation $v=\partial_\alpha u$. The core of the proof is a ``phase transition lemma'' which shows that $v$ possesses exactly one zero between two consecutive zeros of $u$, and that certain auxiliary energy--type functions remain positive throughout the iteration. The arguments of~\cite{Tang2025} rely on the algebraic structure of $f$ and on the dimension $n\ge 3$ at several points, leaving the planar case $n=2$ open for the superlinear problem.

In contrast to the superlinear regime, the nonlinearity
\[
f(u)=u-|u|^{q-2}u,\qquad 1<q<2,
\]
is not locally Lipschitz at $u=0$ and fails to satisfy the classical condition $-\infty<\limsup_{s\to0}f(s)/s<0$ of~\cite{BL1,BL2}. As a consequence, ground states and bound states of~\eqref{eq:sublinear} do not decay exponentially at infinity but instead possess a \emph{finite support radius} $z_u$ at which they vanish together with their first two derivatives. The asymptotic behavior near the boundary is
\[
u(r)\sim C_{q}\,(z_u-r)^{\frac{2}{2-q}}\quad\text{as } r\uparrow z_u .
\]

Despite the singular nature of $f$, the existence of infinitely many radial bound states for~\eqref{eq:sublinear} is known~\cite{Balabane2003, Gazzola2000}, and the uniqueness of the ground state has been established~\cite{PS,Serrin2000}. However, the uniqueness of excited bound states has not been addressed before. The purpose of this paper is to give a complete affirmative answer, valid in all dimensions $n\ge 2$.
   
\medskip
Radial solutions of~\eqref{eq:sublinear} correspond to solutions of the initial value problem
\begin{equation}\label{eqsub}
u''+\frac{n-1}{r}u'+u-|u|^{q-2}u=0,\qquad u(0)=\alpha>0,\; u'(0)=0 .
\end{equation}
Let $\alpha_*=\bigl(\frac{2}{q}\bigr)^{1/(2-q)}$ be the unique positive zero of the potential $F(u)=\frac12 u^2-\frac1q|u|^q$, and let $\alpha^*$ be the unique positive zero of the Pohozaev function $2nF(u)-(n-2)uf(u)$. By the compact support principle~\cite{pucciStrongMaximumPrinciple2004}, a solution may develop a double zero after which it is no longer uniquely determined. For a solution $u(r)$, we define $z_u\in\mathbb{R}^+\cup\{+\infty\}$ as its \emph{first} double zero (with $z_u=+\infty$ if no double zero occurs). As will be proved in Lemma~\ref{lem24}, $u(r)$ is uniquely determined on $r\in(0,z_u)$ by its initial value, and all zeros of $u$ in $(0,z_u)$ are simple.

Our main theorem gives a complete classification of all radial solutions according to the initial value $\alpha$.

\begin{theorem}\label{thm:main}
Let $n\ge 2$, $1<q<2$, and let $u(r,\alpha)$ be the solution of the above initial value problem. There exists a strictly increasing sequence
\[
\alpha_0<\alpha_1<\alpha_2<\cdots,\qquad \alpha_0>\alpha^*,\quad \lim_{k\to\infty}\alpha_k=+\infty,
\]
such that:
\begin{enumerate}
\item[(i)] $u(r,\alpha_0)$ is the unique ground state of~\eqref{eqsub}, extended by zero beyond its support radius $z_u$.
\item[(ii)] For each $k\ge 1$, $u(r,\alpha_k)$ is the unique bound state with exactly $k$ simple zeros in its support $(0,z_u)$. The solution has exactly one critical point between any two consecutive zeros and exactly one critical point behind the last zero; at each critical point one has $|u|>\alpha_*$. The support radius $z_u$ is finite, and
\[
|u(r,\alpha_k)| (z_u-r)^{-\frac{2}{2-q}}\to  C_{q}:=\left( \frac{(2-q)^2}{2q} \right)^{\frac{1}{2-q}}\,\quad\text{as } r\uparrow z_u .
\]
\item[(iii)] If $\alpha=1$, then $u(r,\alpha)\equiv 1$. If $\alpha<\alpha_0$ and $\alpha\neq 1$, then $u(r,\alpha)>0$ for all $r>0$ and oscillates about the constant solution $u\equiv1$.
\item[(iv)] If $\alpha\in(\alpha_k,\alpha_{k+1})$ for some $k\ge 0$, then $u(r,\alpha)$ is a nodal solution with exactly $k+1$ simple zeros in $(0,+\infty)$, and behind its last simple zero it oscillates about $1$ or $-1$.
\end{enumerate}
\end{theorem}

We note that the behaviour of a ground state or bound state beyond its first double zero $z_u$
is also covered by our analysis.
After $z_u$ the solution is no longer uniquely determined by the initial value at $r=0$;
besides the trivial extension $u\equiv0$ on $[z_u,\infty)$,
there exist continuations that vanish on an initial interval $[z_u,s_1]$ and then,
without changing sign, oscillate about $1$ (if the solution stays positive) or about $-1$
(if it stays negative); see Lemma~\ref{lem2.4} for a precise statement.

As a direct corollary of Theorem \ref{thm:main}, we have 
\begin{theorem}
For each positive integer $k$, there exists a unique radial bound state of \eqref{eq:sublinear} in $\mathbb{R}^n$ that possesses precisely $k$ nodes, up to spatial translation and sign reflection.
\end{theorem}

As a by-product of our analysis, we establish the uniqueness of the superlinear model $f(u)=-u+|u|^{p-1}u$ for $n=2$ and $p>1$ via the method of \cite{Tang2025}, recovering a result previously known from \cite{CGY}. Because \cite{Tang2025} relies on the dimension to guarantee the positivity of certain auxiliary functions, we show that for $n=2$, this positivity follows directly from the induction hypothesis. This allows us to extend \cite[Theorems~1 and~2]{Tang2025} to all dimensions $n\ge 2$.

\medskip
 
Our proof follows the general strategy of Tang~\cite{Tang2025}. The central object is the variation $v=\partial_\alpha u$, which satisfies the linearized equation
\[
v''+\frac{n-1}{r}v'+f'(u)v=0,\qquad v(0)=1,\; v'(0)=0.
\]
By locating the zeros of $v$ relative to those of $u$, one shows that the number $\mathcal{N}(\alpha)$ of zeros of $u$ is a non‑decreasing function of $\alpha$, stays constant on intervals of oscillatory solutions, and jumps by at least one when passing through a ground state or a bound state. These properties, together with the existence of solutions having any prescribed number of zeros~\cite{Balabane2003,  Gazzola2000}, force the sequence $\{\alpha_k\}$ to be strictly increasing and unique.

The major technical work consists in proving a ``phase transition lemma'' (Lemma~\ref{lem:phase-transition-sublinear}) which guarantees that in each phase $[c_{i-1},c_i]$ the variation $v$ has exactly one zero $\tau_i$, and that certain energy-type combinations $Q,M,T_2$ remain positive at the right endpoint $c_i$.
For $n\ge 3$, this was established in~\cite{Tang2025} through a series of delicate estimates relying on the Pohozaev identity and the specific algebraic form of $f$. 
The sublinear setting requires two essential modifications:
\begin{itemize}
  \item The compact support forces us to replace the usual exponential decay estimates by an analysis of the behavior of $v$ near the support boundary. In particular, we show that $|v(r)|\to\infty$ as $r\uparrow z_u$ (Lemma~\ref{limitv}), which is crucial for the jump of the zero number at bound states.
  \item The non‑Lipschitz singularity of $f$ at $u=0$ breaks the $C^2$ regularity of $v$ at zeros of $u$. We overcome this difficulty by proving a general uniqueness and continuous dependence result for linear equations with $L^1$ coefficients (Proposition~\ref{lem:singular-linear}) and by carefully checking that all auxiliary functions remain $C^1$ across the simple zeros of $u$.
\end{itemize}
Additionally, we establish a new integral estimate (Lemma~\ref{lem:residual-integral-alln-sublinear})
that holds for all $n\ge 2$ and for a general class of nonlinearities including both the sublinear
and the superlinear cases, thereby generalizing \cite[Lemma~5.4]{Tang2025}.
With this estimate, the only remaining part of the argument that relies on the specific form of the
nonlinearity and on the dimension is the proof of the positivity of $T_2$.
In particular, when $n=2$, the crucial auxiliary function $Q_n$ reduces to $Q_2$, whose positivity
is already guaranteed by the induction hypothesis.
This observation allows us to extend the proof to the planar case without any further restrictions
on the exponent.

The paper is organized as follows.
Section~2 collects basic properties of radial solutions, adapted to the sublinear framework.
Section~3 analyses the first phase and proves that $v$ has exactly one zero before the first zero of $u$.
Section~4 establishes the phase transition lemma, the core of the induction.
Section~5 completes the proof of the main theorem.
Finally, Section~6 indicates the necessary modifications to obtain uniqueness of bound states for the superlinear equation in dimension $n=2$.

\section{Basic properties of radial solutions} 
	Let $u(r)$ be a solution satisfying
\begin{equation}\label{eq2.1}
u'' + \frac{n-1}{r} u' + f(u) = 0, \quad u(0) = \alpha > 0, \quad u'(0) = 0.
\end{equation}
Let $F(u) = \int_0^u f(s) \, ds$.
We first recall the following assumptions on the nonlinearity $f$ that are used in \cite{Tang2025} to study the uniqueness of bound states for superlinear equations.
\begin{enumerate}[label=(C\arabic*), leftmargin=2em, itemsep=0.5ex]
    \item $f \in C^1(\mathbb{R})$, $f(u) < 0$ for $u \in (0,1)$, $f(u) > 0$ for $u \in (1,\infty)$, and $f(-u) = -f(u)$;
    
    \item there is some $\alpha_* > 1$ such that $F(\alpha_*) = \int_0^{\alpha_*} f(s) \, ds = 0$;
    
    \item $\zeta := -f'(0) > 0$;
    
    \item $f'(1) > 0$;
    
    \item there is $\alpha^* > 0$ such that $h(u) := 2nF(u) - (n-2)uf(u) < 0$ for $u \in (0, \alpha^*)$, and $h(u) > 0$ for $u > \alpha^*$;
    
    \item $h_1(u) := uf(u) - 2F(u) \geq 0$;
    
    \item there exists $\tilde{\alpha} > 0$ such that $h_2(u) := (n+2)f(u) - (n-2)uf'(u) < 0$ for $u \in (0, \tilde{\alpha})$, and $h_2(u) > 0$ for $u > \tilde{\alpha}$;
    
    \item $\dfrac{f(u)}{uf'(u)}$ is increasing for $u > 1$;
    
    \item $f''(u) > 0$ for $u > 1$.
\end{enumerate}

For the sublinear nonlinearity $f(u) = u - |u|^{q-2}u$ with $q \in (1,2)$, conditions (C1) and (C3) are not satisfied due to the singularity at 
$u=0$. However, all other structural conditions (C2) and (C4)--(C9) remain valid. Therefore, we replace (C1) and (C3) with the following modified assumptions:
\begin{enumerate} 
    \item[(C1')] 
    \textbf{($C1'_1$)} $f \in C^1(\mathbb{R}\setminus\{0\}) \cap C(\mathbb{R})$; \\
    \textbf{($C1'_2$)} $f(u) < 0$ for $u \in (0,1)$, $f(u) > 0$ for $u \in (1,\infty)$, and $f(-u) = -f(u)$.
    
    \item[(C3')] There exists   $q\in(1,2)$ such that
    \[
    \lim_{u \to 0^+} \frac{f'(u)}{-u^{q-2}} = q-1 > 0.
    \]
\end{enumerate}
As we will show, replacing the classical condition (C3) with (C3') shifts the asymptotic behavior of ground and bound states from exponential decay to compact support. Moreover, the combination of (C1') and (C3') is sufficient to guarantee the unique extension of solutions up to their first double zero.

Note that if $f$ satisfies (C1') and (C3'), then 
by L'Hospital rule 
\[
\lim_{u\to 0^+} \frac{f(u)}{u^{q-1}} = \lim_{u\to 0^+} \frac{f'(u)}{(q-1)u^{q-2}} = -1.
\]
Therefore, $f$ satisfies the compact support principle in \cite{pucciStrongMaximumPrinciple2004}. Hence, any ground state or bound state solution has compact support, and the support radius is a double zero of the solution. To show that 
$u$ can be uniquely extended up to its first double zero, we rely on the fundamental uniqueness result for second-order linear ODEs, whose proof will be given in Appendix.

\begin{proposition} 
\label{lem:singular-linear}
Let $s_0 \ge 0$, $\delta > 0$, $m \ge 0$, and $b, c \in \mathbb{R}$. We impose the following   conditions:
\begin{itemize}
    \item If $s_0 = 0$ and $m > 0$, we assume $b > 0$ and $c = 0$.
    \item If $s_0 > 0$, we assume $\delta < s_0$, ensuring that $[s_0-\delta, s_0+\delta] \subset (0,\infty)$.
\end{itemize}
Define the interval $I=[\max\{0, s_0-\delta\}, s_0+\delta]$.
Consider the initial value problem
\begin{equation}\label{eq:sing}
\bigl(r^{m}u'(r)\bigr)' = r^{m} a(r) u(r), \qquad u(s_0)=b, \quad u'(s_0)=c,
\end{equation}
where the coefficient $a \in L^1(I)$ is continuous on $I \setminus \{s_0\}$. Then the following properties hold:
\begin{enumerate}
    \item[\textup{(i)}] There exists a unique solution $u$ to \eqref{eq:sing} satisfying the regularity $u \in C^1(I) \cap C^2\bigl(I \setminus \{s_0\}\bigr)$.

    \item[\textup{(ii)}] Let $\bar{a} \in L^1(I)$ be continuous on $I \setminus \{s_0\}$, and let $\bar{u}$ be the corresponding solution to \eqref{eq:sing} with coefficient $\bar{a}$ and the same initial data $(b,c)$. If $\|\bar{a} - a\|_{L^1(I)}$ is sufficiently small, then
    \[
    \|\bar{u} - u\|_{C^1(I)} \le C \|\bar{a} - a\|_{L^1(I)},
    \]
    where the constant $C > 0$ depends only on $m, s_0, \delta, b, c$, and $\|a\|_{L^1(I)}$.
\end{enumerate}
\end{proposition}

\begin{lemma}\label{lem24}
	Assume (C1') and (C3'). Let $u$ be a solution of \eqref{eq2.1}.
	If $s_0\geq 0$ is such that $u(s_0)\neq 0$ or $u'(s_0)\neq 0$, then there is $\delta>0$ such that 
	$u(s)$ can uniquely extend to the interval $[\max\{0,s_0-\delta\}, s_0+\delta]$ as a classical $C^2$ solution.
\end{lemma}
\begin{proof} 
	
	If $u(s_0)\neq 0$, then $f$ is locally Lipschitz in a neighborhood of $u(s_0)$. Hence the uniqueness of the local extension follows from the standard ODE theory. When $s_0=0$, the same conclusion follows from the usual integral formulation of the radial equation together with $u'(0)=0$.
	
	If $u(s_0) = 0$ but $u'(s_0) \neq 0$, then necessarily $s_0>0$. Let
	$
	d:=u'(s_0)\neq0 .
	$
	Assume that \(u\) and \(\bar u\) are two extensions with the same data
	\[
	u(s_0)=\bar u(s_0)=0,\qquad u'(s_0)=\bar u'(s_0)=d .
	\]
	Choosing \(\delta>0\) sufficiently small, with \(\delta<s_0\), both \(u\) and \(\bar u\) have the same sign as
	\(d(r-s_0)\) on \(0<|r-s_0|<\delta\). Moreover, there exists \(c>0\) such that
	\[
	|u(r)|\ge c|r-s_0|,\qquad |\bar u(r)|\ge c|r-s_0|
	\quad\text{for }0<|r-s_0|<\delta .
	\]
	Let \(w=u-\bar u\). Then
	\[
	(r^{n-1}w')'
	=
	-r^{n-1}\bigl(f(u)-f(\bar u)\bigr)
	=
	r^{n-1}A(r)w,
	\]
	where
	\[
	A(r):=-\int_0^1 f'\bigl(\bar u(r)+\theta(u(r)-\bar u(r))\bigr)\,d\theta .
	\]
	By (C3') and the oddness of \(f\), after possibly decreasing \(\delta\), we have
	\[
	|A(r)|\le C|r-s_0|^{q-2}
	\quad\text{for }0<|r-s_0|<\delta .
	\]
	Since \(q>1\), it follows that \(A\in L^1_{\mathrm{loc}}\bigl((s_0-\delta,s_0+\delta)\bigr)\). Applying Proposition~\ref{lem:singular-linear} to \(w\), with zero initial data
	$
	w(s_0)=w'(s_0)=0,
	$
	gives \(w\equiv0\) on \([s_0-\delta,s_0+\delta]\). Therefore the extension is unique.
	
	Finally, since \(f(u(r))\) is continuous and \(s_0>0\), the equation
	\[
	u''+\frac{n-1}{r}u'+f(u)=0
	\]
	implies that the extension is a classical \(C^2\) solution on
	\([\max\{0,s_0-\delta\},s_0+\delta]\).
\end{proof}

\begin{definition}
	Under assumptions (C1') and (C3'), for any solution \(u\), we define
	\(z_u\in \mathbb R^+\cup\{+\infty\}\) as its first double zero, with the
	convention that \(z_u=+\infty\) if no such zero exists.
\end{definition}
By Lemma~\ref{lem24}, the solution is uniquely determined on \((0,z_u)\)
by its initial value. Moreover, every zero of \(u\) in \((0,z_u)\) is simple:
\begin{equation}\label{eqzero}
u'(z)\neq0\quad\text{whenever }u(z)=0,\ z\in(0,z_u).
\end{equation}

We define the energy function
\begin{equation*}
	E(r):=\frac{u'(r)^2}{2}+F(u(r)),
\end{equation*}
with the derivative
	\begin{equation}\label{eprime}
	E'(r) = u' u'' + f(u) u' = u' (u'' + f(u))  = -\frac{n-1}{r} u'^2 \leq 0.
	\end{equation}

Then we have the following results.
\begin{proposition}\label{pro2.5}
	Assume (C1') (C2) and (C3').
Let $u(r)$ be a solution of \eqref{eq2.1} with $\alpha\neq 1$. The following statements hold.
	\begin{enumerate}
		\item[(i)]  $E(r)$ is strictly decreasing in $(0,z_u)$.
		\item[(ii)] If \(0<\bar r<\hat r<z_u\) and \(|u(\hat r)|=|u(\bar r)|\), then $|u'(\hat{r})| < |u'(\bar{r})|$.
		\item[(iii)] For $r\in (0, z_u)$, we have $|u(r)| <\alpha$  if $\alpha > 1$
		\item[(iv)]  If $0 < \alpha \leq  \alpha_*$, then $z_u = +\infty$ and  $\min\{\alpha, \alpha'\} < u(r) < \max\{\alpha, \alpha'\}$ for all $r > 0$, where $\alpha'\neq \alpha$ is the unique nonnegative point such that $F(\alpha') = F(\alpha)$.
	\end{enumerate}
\end{proposition}
\begin{proof}
(i)	For \eqref{eprime}, $E'(r) \leq 0$ for $r \in (0, z_u)$.
	If $E(r)$ were not strictly decreasing, then there would exist $r_1 < r_2$ such that $E(r)$ is constant on $[r_1, r_2]\subset(0,z_u)$, so $E'(r) = 0$ on $[r_1, r_2]$, hence $u'(r) = 0$ on $[r_1, r_2]$. By \eqref{eqzero}, $u(r)=c\neq 0$ on $[r_1,r_2]$. Substituting into the equation gives $f(c) = 0$, and hence $c = 1$.
	This, together with the uniqueness of ODE, also implies that
	$u(r) \equiv c=1$ on $[0,+\infty)$.
	 Thus we get a contradiction, and hence $E(r)$ is strictly decreasing.
	
	(ii) Let $z_u\geq \hat{r} > \bar{r} > 0$ and $|u(\hat{r})| = |u(\bar{r})|$. We prove $|u'(\hat{r})| < |u'(\bar{r})|$.
	Since $E(r)$ is strictly decreasing, $E(\hat{r}) < E(\bar{r})$. Also, $F(u)$ is even, so $F(u(\hat{r})) = F(u(\bar{r}))$. Then 
	\[
	\frac{1}{2} u'(\hat{r})^2 = E(\hat{r}) - F(u(\hat{r})) < E(\bar{r}) - F(u(\bar{r})) = \frac{1}{2} u'(\bar{r})^2,
	\]
	so $|u'(\hat{r})| < |u'(\bar{r})|$.
	
(iii) Assume that $\alpha > 1$.  For $r \in(0,z_u)$, since $E(0) = F(\alpha)$ and  $E(r) < E(0)$, we have $F(u(r)) \leq E(r) < E(0)=F(\alpha)$. Note that $F(u)$ is increasing on $[1, \infty)$ (because $f(u) > 0$ for $u > 1$). If for some $r\in(0,z_u)$,  $|u(r) |\geq \alpha$, then $F(u(r)) =F(|u(r)|)\geq F(\alpha)$, a contradiction. Hence $|u(r)|  < \alpha$ for $r\in(0,z_u)$.
	
(iv) By (i) and $E(0)\leq 0$, we have 
\(F(u(r))\leq E(r)<0
\) for each  $r\in (0, z_u)$. 
 
Since \(F(0)=0\), the inequality \(F(u(r))<0\) excludes zeros of \(u\) for
\(r>0\). As \(u(0)=\alpha>0\), we have \(u(r)>0\) for all \(r>0\).
 If \(z_u<+\infty\), then \(E(z_u)=0\), contradicting \(E(r)<0\) for \(r>0\) and
the monotonicity of \(E\). Hence \(z_u=+\infty\).
\end{proof}

\begin{proposition}\label{pro2.6} Assume (C1'), (C2) and (C3').
	Let  $u$ be the solution of \eqref{eq2.1}.
	Then the following hold:
	\begin{enumerate}
		\item[(i)] If $u$ is a ground state or a nodal solution, then $\alpha > \alpha_*$.
		
		\item[(ii)] If $u$ is a ground state, then $u'(r) < 0$ for all $r \in (0, z_u)$.
		
		\item[(iii)] If $u$ is a nodal solution with $k \geq 1$ zeros before $z_u$: $0<z_1 < z_2 < \cdots < z_k<z_u$, then $u$ has $k$ critical points in $[0, z_k]$, labeled as $0 = c_0 < c_1 < c_2 < \cdots < c_{k-1}$ with $c_i \in (z_i, z_{i+1})$ for $1 \leq i \leq k-1$, and
		\[
		u(0) > |u(c_1)| > |u(c_2)| > \cdots > |u(c_{k-1})| > \alpha_*,
		\]
		Moreover, if $u$ is a bound state with $k$ zeros before $z_u$, then there is a unique critical point $c_k \in (z_k,z_u)$ with $|u(c_{k-1})| > |u(c_k)| > \alpha_*$.
		\item[(iv)]If  $u$ is a ground state or a bound state, then  \[
		\frac{u'(r) }{|u(r)|^{\frac{q}{2}-1}u(r)} \to -\sqrt{\frac{2}{q}},\quad
		|u(r)| (z_u-r)^{-\frac{2}{2-q}}\to C_{q}:=\left( \frac{(2-q)^2}{2q} \right)^{\frac{1}{2-q}} \text{ as } r\uparrow z_u.\]  
	\end{enumerate}
\end{proposition}

\begin{proof}
  (i) follows directly from (iv) of Proposition \ref{pro2.5}.
	
	Now we consider (ii). Let $u$ be a ground state. Then $u(r)>0$ in $(0,z_u)$, $u'(0)=0$ and $u''(0) = -f(\alpha)/n$. Since $\alpha > \alpha_* > 1$, we have $f(\alpha) = \alpha - \alpha^{q-1} > 0$. So $u''(0) < 0$, hence $u'$ becomes negative immediately. Suppose by contradiction that there exists a first point $\tilde{c} \in(0,z_u)$ such that $u'(\tilde{c}) = 0$. 
	Then $u'<0$ on $(0, \tilde{c})$. At $\tilde{c}$, we have $u''(\tilde{c}) \geq 0$. 
	Since $F(u(\tilde{c}))=E(\tilde{c}) > E(z_u) = 0$, we know $F(u(\tilde{c}))>0$, and hence $u(\tilde c)>\alpha_*$. Therefore, $u''(\tilde c)=-f(u(\tilde c))<0$, which is a contradiction.
	
	For (iii), let $u$ be a nodal solution with $k$ zeros $0<z_1 < z_2 < \cdots < z_k<z_u$. The zeros are simple by \eqref{eqzero}. By the same proof as (ii), $u' < 0$ on $(0, z_1)$.
	On $(z_1, z_2)$, there must be at least one critical point of $u$. Then for any
	$c_1\in(z_1,z_2)$ such that $u'(c_1)=0$, we have $F(u(c_1))=E(c_1)>E(z_2)>0=F(\alpha_*)$. Hence $|u(c_1)|>\alpha_*$, and $u''(c_1)=-f(u(c_1))>0$. This implies that $c_1$ is the only critical point of $u$ in $(z_1,z_2)$. Noting that $F(u(c_1))=E(c_1)<E(0)=F(\alpha)$, by $|u(c_1)|>\alpha_*$ and  (i): $\alpha>\alpha_*$, we deduce $|u(c_1)|<\alpha=u(0)$.	
	By induction, we obtain the sequence of critical points $c_i$ with $|u(c_{i-1})| > |u(c_i)| > \alpha_*$.
	
	If $u$ is a bound state, by $u(z_k)=u(z_u)=0$, there must be at least one critical point of $u$. Then for any
	$c_k\in(z_k,z_u)$ such that $u'(c_k)=0$, we have $F(u(c_k))=E(c_k)>E(z_u)=0=F(\alpha_*)$. Hence $|u(c_k)|>\alpha_*$, and $u''(c_k)=-f(u(c_k))<0$. This implies that $c_k$ is the only critical point of $u$ in $(z_k,z_u)$. Noting that $F(u(c_k))=E(c_k)<E(c_{k-1})=F(u(c_{k-1}))$, by $|u(c_{k})|>\alpha_*$ and $|u(c_{k-1})|>\alpha_*$, we deduce $|u(c_k)|<|u(c_{k-1})|$.	

(iv) Without loss of generality, we may assume that  
$u(r)>0$  when $r<z_u$ is close to $z_u$.
Then 
by $(r^{n-1}u')'=-r^{n-1}f(u(r))>0$, we know that 
$r^{n-1}u'<0$ is strictly increasing in $r$.  
Integrating \eqref{eprime} from $r$ to $z_u$ gives
\[
 E(r)=\int_r^{z_u} \frac{n-1}{s} (u'(s))^2 ds\leq r^{2n-2}|u'(r)|^2\int_r^{z_u} \frac{n-1}{s^{2n-1}}   ds=o(|u'(r)|^2) .
\]
 Since $z_u$ is a double zero, we have $E(z_u)=0$.  Hence, $E(r)\geq 0$ for $r\leq z_u$.
 Therefore, $E(r) = o(|u'(r)|^2)$ as $r \uparrow z_u$. On the other hand,
  condition (C3') implies that $F(u) =-\frac{1}{q} u^q +o(u^q)$ as $r \uparrow z_u$. Hence,
\[
  (\frac12 +o(1)) |u'(r)|^2 =   (\frac{1}{q} +o(1)) u(r)^q   \text{ as } r \uparrow z_u.
\]
Since $u(r)>0$, $u'(r)<0$ for $r$ close to $z_u$,
we may write the equality as
$
\frac{2}{2-q}\bigl(u^{1-\frac q2}\bigr)'
=
-\sqrt{\frac2q}+o(1),
$
and integrating from \(r\) to \(z_u\) yields the desired asymptotic formula.
\end{proof}

\begin{proposition}\label{pro2.7}
	 Suppose (C1'), (C2), and (C4) hold and $u(r) \not\equiv 1$ is the solution of \eqref{eq2.1}. If $E(\bar{r}) \le 0$ and $u(\bar{r}) > 0$ at some $\bar{r} \ge 0$, then $u(r) \in (0, \alpha_*)$ and it oscillates about $u \equiv 1$ in $(\bar{r}, \infty)$: There are a sequence of critical points of $u$, $\widetilde{c}_1 < \widetilde{c}_2 < \cdots$, along which $\alpha_* > u(\widetilde{c}_1) > u(\widetilde{c}_3) > \cdots > 1$, and $0 < u(\widetilde{c}_2) < u(\widetilde{c}_4) < \cdots < 1$.
If $E(\bar{r}) \le 0$ and $u(\bar{r}) < 0$ at some $\bar{r} \ge 0$, then the same can be said for $-u$.
\end{proposition}
\begin{proof}
	Assume $E(\bar r)\leq 0$ and $u(\bar r) > 0$.
If $E(r) = 0$ for some $r > \bar{r}$. By monotonicity, we must have $E(s) = 0$ and $E'(s) = 0$ for all $s \in [\bar{r}, r]$. The condition $E'(s) = 0$ yields $u'(s) = 0$, meaning $u(s)$ is constant on this interval.   This implies $u \equiv 0$ or $\pm 1$. Since $u(\bar{r}) > 0$, we must have $u\equiv 1$, which contradicts the assumption.

Observe that $F(u)$ is an even function, and $F(u) < 0$ if and only if $|u| \in (0, \alpha_*)$. This implies $|u(r)| \in (0, \alpha_*)$. Combined with the fact that $u(\bar{r}) > 0$, continuity ensures that $u(r) \in (0, \alpha_*)$ for all $r \in (\bar{r}, \infty)$.
Given that the solution remains strictly positive,   our assumption (C1') reduces to the standard condition (C1) in \cite{Tang2025}. Thus, the conclusion follows by an identical application of the argument in \cite[Proposition 2.3]{Tang2025}
\end{proof}
 
After the first  double zero, the solution is not unique. But we can classify the solutions by the following  cases.

\begin{lemma}\label{lem2.4}
Assume (C1') (C2)  (C3') and (C4). If there is $s_0 \ge 0$ such that $u(s_0) = u'(s_0) = 0$, then one of the following holds:
\begin{enumerate}
    \item $u(r) \equiv 0$ for $r \in [s_0, \infty)$.
    \item There is $s_1 \ge s_0$ such that $u \equiv 0$ on $[s_0, s_1]$ and $u$ does not change sign on $(s_1, \infty)$. More precisely, $0 < |u(r)| < \alpha_*$ for $r > s_1$, and $u$ oscillates about $1$ or $-1$.
\end{enumerate}	
\end{lemma}
\begin{proof}
	Let
	$
	s_1:=\sup\{s\ge s_0:\ u\equiv0\text{ on }[s_0,s]\}.
	$
	If \(s_1=+\infty\), then \(u\equiv0\) on \([s_0,\infty)\), and (i) holds.
	
	Assume \(s_1<+\infty\). Then \(u(s_1)=u'(s_1)=0\), and \(u\) is not identically
	zero to the right of \(s_1\). Choose \(\bar r>s_1\) arbitrarily close to \(s_1\)
	such that \(u(\bar r)\neq0\). Since \(E(s_1)=0\) and \(E\) is nonincreasing, we
	have \(E(\bar r)\le0\). Applying Proposition~\ref{pro2.7} to \(u\) if
	\(u(\bar r)>0\), and to \(-u\) if \(u(\bar r)<0\), gives
	$
	0<|u(r)|<\alpha_*
	\ \text{for }r>\bar r,
	$
	and \(u\) oscillates about \(1\) or \(-1\). Since \(\bar r\downarrow s_1\) is
	arbitrary, the conclusion holds on \((s_1,\infty)\).
\end{proof}

By Lemma~\ref{lem2.4}, the set of double zeros of a solution is either a single point, a closed interval, or the empty set.

\begin{remark}
	If $u$ is a ground state or a bound state, then its first double zero coincides with its support radius, i.e., $z_u = R_u$. Since the non-trivial structure of the solution is entirely contained within $(0, z_u)$, we restrict our attention to this interval. 
\end{remark}

\begin{proposition}\label{pro2.8}
Assume (C1') (C2)  (C3') and (C4).	Let $u(r)$ be a solution of \eqref{eq2.1}.  
	\begin{enumerate}
		\item[(i)] If $\alpha \in(0, \alpha_*]$ and $\alpha \neq 1$, then $u$ is positive and oscillates around $1$ in $(0,\infty)$.
		
		\item[(ii)] A nodal solution $u$ has only finitely many sign changes in $(0,z_u)$.
		
		\item[(iii)] If $u$ is a ground state or a bound state, then $E(r) > 0$ for all $r \in(0,z_u)$. If $u$ is a nodal solution with largest zero point $z_k<z_u$, then $E(r) > 0$ for all $r \in(0,z_k]$.
		
		\item[(iv)] In $(0,z_u)$, a positive solution $u$ is either a ground state, or an oscillatory function that oscillates about $1$; a nodal solution $u$ is either a bound state, or an oscillatory function that oscillates about $1$ or $-1$ after its last zero.
	\end{enumerate}
\end{proposition}

\begin{proof}
	(i) follows directly from Proposition \ref{pro2.7}, since $E(0) \leq 0$.
	
	(ii) Suppose for contradiction that a nodal solution $u$ has infinitely many zeros. Let these zeros be $z_1 < z_2 < \cdots$ with $\lim_{i\to\infty} z_i = z\in(0,z_u]$.
	Let $c_i\in(z_i,z_{i+1})$ be a critical point with $F(u(c_i))> E(z_{i+1})\geq0=F(\alpha_*)$. So $|u(c_i)|>\alpha_*$ and $u''(c_i)=-f(u(c_i))\neq 0$. This implies that $c_i$ is the only critical point of $u$ in $(z_i,z_{i+1})$.
 Then $u$ is monotone in $(c_{i-1},c_i)$, and we can
	choose $b_i, \bar{b}_i \in (c_{i-1}, c_i)$ such that 
	\[b_i < \bar{b}_i,\quad  |u(b_i)| = |u(\bar{b}_i)| = \alpha_*,\ \text{and}\ |u(r)|\leq \alpha_*\ \text{for}\  r\in [b_i,\bar b_i].\]
By Proposition \ref{pro2.5} (ii), we have 
\[
2\alpha_* = |u(b_i)-u(\bar b_i)|\leq \max_{r\in(b_1,\bar b_1)}|u'(r)||\bar b_i - b_i|.
\]
Therefore, $\bar b_i - b_i $ is bounded below for all $i$. This implies that $z = z_u = +\infty$.
 The proof of  (ii) then follows exactly the same argument as in \cite[Proposition 2.4(ii)]{Tang2025}. This is because the proof relies solely on the continuity and sign properties of $f(u)$ and $F(u)$, and does not involve any derivatives of $f$. Consequently, the non-Lipschitz singularities at $u=0$ 
 relaxed in (C1')  and (C3') are entirely irrelevant to this argument.
	
	(iii) follows directly from Proposition \ref{pro2.5} (i) and the fact that $E(z_u) =0$ and  $E(z_k)>0$ since $z_k$ is a simple zero.
	
	(iv) Let $u$ be a nodal solution with $z_u=+\infty$. According to (ii), we assume its last zero point is $z_k$. If $E(r) \leq 0$ at some $r > z_k$, then Proposition \ref{pro2.7} implies that $u$ oscillates about $1$ or $-1$ in $(z_k, \infty)$. If $E(r) > 0$ for all $r > z_k$, then $u(r)$ has no critical points at which $|u| \leq \alpha_*$. Thus $u$ has at most one critical point in $(z_k, \infty)$, and is monotone for sufficiently large $r$. Clearly, $u_\infty = \lim_{r \to \infty} u(r)$ is finite in view of Proposition \ref{pro2.5} (iii), and $f(u_\infty) = 0$. Since $E(r) > 0$ for all $r > 0$, $u_\infty$ is neither $1$ nor $-1$. Thus $u_\infty = 0$ and the compact support principle in \cite{pucciStrongMaximumPrinciple2004} implies that $u$ is a bound state with compact support. 
\end{proof}
\subsection{Phases and labels}
Under the assumption of Proposition \ref{pro2.8},
if \(u\) is a bound state   with $z_u<+\infty$,  then it has exactly \(k\) simple zeros before $z_u$, \(z_{1}<\cdots<z_{k}<z_u\). It follows from Proposition \ref{pro2.6} (iii) that \(u\) has \(k+1\) critical points \(0=c_{0}<c_{1}<\cdots<c_{k}\). We decompose \((0, z_u)\) into \(k\) phases and a semi-tail phase separated by these critical points:
\begin{equation}\label{eq2.9}
	(0,z_u)=\underbrace{(0,c_{1}]}_{\text{Phase }1}\cup\underbrace{(c_{1},c_{2}]}_{\text{Phase }2}\cup\cdots\cup \underbrace{(c_{i-1},c_{i}]}_{\text{Phase }i}\cup\cdots\cup\underbrace{(c_{k-1},c_{k}]}_{\text{Phase }k}
	\cup\underbrace{(c_{k},z_u)}_{\text{Semi-tail phase}}.
\end{equation}
 
 If \(u\) is a nodal solution with $z_u=+\infty$, then it has exactly \(k\) simple zeros and oscillates in \((z_{k},\infty)\). We make a similar decomposition that differs from 
 \eqref{eq2.9} only in the last two intervals:
 \begin{equation*} 
 	(0,\infty)=\underbrace{(0,c_{1}]}_{\text{Phase }1}\cup\underbrace{(c_{1},c_{2}]}_{\text{Phase }2}\cup\cdots\cup \underbrace{(c_{i-1},c_{i}]}_{\text{Phase }i}\cup\cdots\cup\underbrace{(c_{k-1},z_{k}]}_{\text{Semi-phase }k}
 	\cup\underbrace{(z_{k},\infty)}_{\text{Tail phase}}.
 \end{equation*}
The function \(u\) decreases in Phase \(i\) if \(i\) is odd, and increases in Phase \(i\) if \(i\) is even.

Within a single phase \((c_{i-1},c_{i}]\), \(u\), \(f(u)\), and \(F(u)\) all change signs. We label the points where \(F(u)=0\) or \(f(u)=0\) besides \(r=z_{i}\) as follows:
\begin{equation*} 
	\begin{cases}
		|u(b_{i})|=|u(\overline{b}_{i})|=\alpha_{*}, \ \ \ |u(r_{i})|=|u(\overline{r}_{i})|=1, \\
		c_{i-1}<b_{i}<r_{i}<z_{i}<\overline{r}_{i}<\overline{b}_{i}<c_{i}.
	\end{cases}
\end{equation*}
We use these numbers to decompose Phase \(i\) into subsets in which \(u\), \(f(u)\), and \(F(u)\) remain the same sign. For an odd number \(i\),
\begin{align*} 
	(c_{i-1},c_{i}] ={}& \overbrace{\underbrace{(c_{i-1},b_{i})}_{f, F>0}\cup\underbrace{\{b_i\}}_{F=0}\cup\underbrace{(b_{i},r_{i})}_{f>0,F<0}
	\cup\underbrace{\{r_i\}}_{f=0}\cup\underbrace{(r_{i},z_{i})}_{f,F<0}}^{u>0}
	\cup\underbrace{\{z_i\}}_{u,f,F=0} 
	\cup 
	 \overbrace{\underbrace{(z_i,\overline r_i)}_{f>0, F<0}\cup\underbrace{\{\overline r_i\}}_{f=0}\cup\underbrace{(\overline r_i,\overline b_i)}_{f,F<0}
		\cup
		\underbrace{\{\overline b_i\}}_{F=0} \cup\underbrace{(\overline b_i,c_i)}_{f<0,F>0}}^{u<0},
\end{align*}
and for an even number \(i\),
\begin{align*} 
	(c_{i-1},c_{i}] =& \overbrace{\underbrace{(c_{i-1},b_{i})}_{f<0, F>0}\cup\underbrace{\{b_i\}}_{F=0}\cup\underbrace{(b_{i},r_{i})}_{f,F<0}
	\cup\underbrace{\{r_i\}}_{f=0}\cup\underbrace{(r_{i},z_{i})}_{f>0,F<0}}^{u<0}
\cup\underbrace{\{z_i\}}_{u,f,F=0} 
\cup 
\overbrace{\underbrace{(z_i,\overline r_i)}_{f, F<0}\cup\underbrace{\{\overline r_i\}}_{f=0}\cup\underbrace{(\overline r_i,\overline b_i)}_{f>0,F<0}
	\cup
	\underbrace{\{\overline b_i\}}_{F=0} \cup\underbrace{(\overline b_i,c_i)}_{f,F>0}}^{u>0}.
\end{align*}

\subsection{Positivity of energy-type functions}

We first consider the case \(n\ge3\). The Pohozaev function associated with equation \eqref{eq2.1} is defined by
\begin{equation*} 
P(r) := 2r^n E(r) + (n - 2)r^{n-1}uu'
= r^n \left[ u'^2 + 2F(u) \right] + (n - 2)r^{n-1}uu'.
\end{equation*}
We will establish the positivity of \( P(r) \) and two related functions
\begin{equation}\label{P1}
P_1(r) := r^n \left[ u'^2 + uf(u) \right] + (n - 2)r^{n-1}uu',
\end{equation}
and
\begin{equation*}
P_2(r) := r^n \left[ u'^2 + \frac{n - 2}{n}uf(u) \right] + (n - 2)r^{n-1}uu'.
\end{equation*}
We have
\begin{align}
    P'(r) &= r^{n-1} h(u), \quad \text{where} \quad h(u) = 2n F(u) - (n-2)u f(u) \text{ is given in (C5)}, \label{eq:P_prime} \\
    P_2'(r) &= -\frac{r^n u'}{n} h_2(u), \quad \text{where} \quad h_2(u) = (n+2) f(u) - (n-2) u f'(u)\text{ is given in (C7)}. \label{eq:P2_prime}
\end{align}
 
\begin{proposition}\label{pro2.9} Assume (C1') (C2)  (C3') hold and $n\geq 3$. Let \(u(r)\) be the solution of (2.1). Then the following statements are valid for \(r\in(0,z_{k}]\) if \(u\) is a nodal solution with the largest zero \(z_{k}<z_u\), or for all \(r\in(0,z_u)\) if \(u\) is a ground state or a bound state.
	
	(i) Assume (C5) holds. Then \(P(r)>0\); in addition, if \(u\) is a ground state or a nodal solution, then \(u(0)>\alpha^{*}\).
	
	(ii) If (C5) and (C6) hold, then \(P_{1}(r)>0\), and \(\omega^{\prime}(r)>0\) as long as \(u(r)\neq 0\), where \(\omega(r):=-ru^{\prime}(r)/u(r)\).
	
	(iii) If (C7) holds, \(P_{2}(r)>0\) and \(-[P(r)/r^{n}]^{\prime}>0\).
\end{proposition}

\begin{proof}
	The proof is the same as that of \cite[Proposition~2.5]{Tang2025}. The only
	point where the regularity of \(f\) at the origin is relevant is the derivation of
	\(P_2'\). Under (C3'), the function
	\[
	h_2(u)=(n+2)f(u)-(n-2)uf'(u)
	\]
	has a continuous extension to \(u=0\), with \(h_2(0)=0\). Indeed,
	\(f(u)=-u^{q-1}+o(u^{q-1})\) and \(u f'(u)=-(q-1)u^{q-1}+o(u^{q-1})\) as
	\(u\to0^+\). Hence \(P_2\) is \(C^1\) across simple zeros of \(u\), and the
	argument of \cite[Proposition~2.5]{Tang2025} applies unchanged.
\end{proof}

We next record the corresponding planar version. In dimension \(n=2\), the
functions above reduce to
\[
P(r)=r^2\left[u'(r)^2+2F(u(r))\right]=2r^2E(r),
\quad
P_1(r)=r^2\left[u'(r)^2+u(r)f(u(r))\right],
\]
and
\[
P_2(r)=r^2u'(r)^2.
\]
Moreover, when \(n=2\),
$
h(u)=4F(u),\  h_2(u)=4f(u).
$
Hence (C5) follows from (C1') and (C2), with \(\alpha^*=\alpha_*\), and (C7)
follows from the sign condition on \(f\), with \(\widetilde\alpha=1\).

\begin{proposition}[Planar version]\label{pro2.9-planar}
	Assume (C1') (C2)  (C3') and \(n=2\). Let \(u(r)\) be the solution of \eqref{eq2.1}. Then the following
	statements are valid for \(r\in(0,z_k]\) if \(u\) is a nodal solution with largest
	zero \(z_k<z_u\), or for all \(r\in(0,z_u)\) if \(u\) is a ground state or a bound state.
	
	\begin{enumerate}
		\item[(i)] \(P(r)>0\). In addition, if \(u\) is a ground state or a nodal solution,
		then \(u(0)>\alpha^*=\alpha_*\).
		
		\item[(ii)]Assume (C6). Then \(P_1(r)>0\), and \(\omega'(r)>0\) as long as \(u(r)\neq0\), where
		$\omega(r):=-ru^{\prime}(r)/u(r).$

		\item[(iii)]   \(P(r)/r^2\) is strictly decreasing on $(0,z_u)$.
	\end{enumerate}
\end{proposition}

\begin{proof}
	Since \(n=2\), we have
	$
	P(r)=2r^2E(r).
	$
	By Proposition \ref{pro2.8} (iii), \(E(r)>0\) on the relevant interval. Hence
	$
	P(r)>0.
	$
	If \(u\) is a ground state or a nodal solution, Proposition \ref{pro2.6} (i) gives
	$
	u(0)>\alpha_*=\alpha^*.
	$
	This proves (i).
	
	For (ii), note that
	$
	P_1(r)=P(r)+r^2\bigl(uf(u)-2F(u)\bigr)\geq P(r)>0,
	$
	 by (C6) and  
	 \(P(r)>0\).
	Using the identity
	\[
	\omega'(r)=\frac{P_1(r)}{r^{n-1}u(r)^2},
	\]
	with \(n=2\), we get
	$
	\omega'(r)>0
	$
	as long as \(u(r)\neq0\). This proves (ii).
	
Finally, for (iii), in the two-dimensional case, we have 
	  \(P(r)/r^2=2E(r)\),
 It follows from equation \eqref{eprime} and Proposition \ref{pro2.5} (i)  that it is strictly decreasing on $(0,z_u)$.  
\end{proof}

\subsection{The concavity}
This subsection is not used in the proof of the main theorem. It is included only
to record a qualitative concavity property of ground states and nodal solutions;
the auxiliary-function positivity needed below follows from
Propositions~\ref{pro2.9} and~\ref{pro2.9-planar}.
\begin{proposition}
	Assume that (C1') (C2) (C3'), (C4)--(C6) and (C8) hold and $n\geq 2$.
	A ground state \( u \) of \eqref{eq2.1} changes concavity at a unique point of inflection in \((0, z_u)\). For a nodal solution \( u \) of (2.1), there is a unique point of inflection between each critical point \( c_{i-1} \) and its nearest zero \( z_i > c_{i-1} \).
\end{proposition}

\begin{proof}
For odd $i$, $u' < 0$ and $u > 0$ on $(c_{i-1}, z_i)$, which implies $u'' = -\frac{n-1}{r}u' - f(u) > 0$ whenever $0 \le u < 1$. Conversely, for even $i$, $u' > 0$ and $u < 0$ on $(c_{i-1}, z_i)$, yielding $u'' = -\frac{n-1}{r}u' - f(u) < 0$ whenever $-1 < u \le 0$. 
Hence, the argument is confined to the regions where $|u| > 1$. The   conditions (C1') and (C3') are therefore bypassed, and the proof follows verbatim from \cite[Proposition 2.6]{Tang2025}.

Furthermore, the argument is independent of the dimension restriction $n \ge 3$. In \cite[Proposition 2.6]{Tang2025}, conditions (C5) and (C6) are invoked only to ensure the properties of $\omega$, which hold for $n \ge 3$ by Proposition \ref{pro2.9} (ii) and for $n = 2$ by Proposition \ref{pro2.9-planar} (ii).
     
\end{proof}

We note that, for a nodal solution \(u\) with zeros \(z_{1} < \cdots < z_{k}\), its convexity in the intervals \((z_{i}, c_{i})\) for \(1 \leq i \leq k-1\) and \((z_{k}, \infty)\) is not addressed in this proposition.

\section{Analysis in Phase 1}

In this section we study the variation
\[
v(r)=\partial_\alpha u(r,\alpha)
\]
in the first phase. Recall that
 \(u=u(r,\alpha)\) is the solution of
\[
u''+\frac{n-1}{r}u'+f(u)=0,\qquad
u(0)=\alpha>0,\quad u'(0)=0.
\]

\begin{lemma}
\label{lem-global-v}
Suppose that (C1'), (C2), (C3'), and (C4) are satisfied with $n \ge 2$. 
Let $u(r,\alpha)$ solve \eqref{eq2.1} with $u(0)=\alpha > 0$. Then the map $(r,\alpha) \mapsto u(r,\alpha)$ is continuously differentiable on $[0, z_u) \times (0, +\infty)$, and its $\alpha$-derivative $v = \partial_\alpha u$ satisfies
\begin{equation}\label{eq:v-equation}
(r^{n-1}v')' + r^{n-1}f'(u)v = 0, \qquad v(0) = 1, \quad v'(0) = 0.
\end{equation}
Additionally, $v \in C^1([0, z_u)) \cap C^2([0, z_u) \setminus \mathcal{Z})$, where $\mathcal{Z}$ is the zero set of $u$ in $[0, z_u)$, and $v$ has no double zeros in this interval.
\end{lemma}
\begin{proof}
	By Proposition \ref{pro2.8} (ii), $\mathcal{Z}$ is a finite set.
Fix  any $R \in (0, z_u)\setminus \mathcal{Z}$. 
Since all zeros of $u$ in $[0, R]$ are simple, the same arguments as in Lemma \ref{lem2.4} show that $f'(u) \in L^1([0, R])$. 
By Proposition \ref{lem:singular-linear} (i), the variational equation \eqref{eq:v-equation} admits a unique solution $v \in C^1([0,R]) \cap C^2([0,R] \setminus \mathcal{Z})$. Since $R < z_u$ can be arbitrary close to $z_u$, $v$ is well-defined on the entire interval $[0, z_u)$. By uniqueness, $v$ admits only simple zeros in 
$[0, z_u)$.
We claim that 
\begin{equation}\label{eq:fprime-continuity}
\|f'(\tilde{u}) - f'(u)\|_{L^1([0, R])} \to 0 \quad \text{as} \quad \|\tilde{u} - u\|_{C^1([0, R])} \to 0.
\end{equation}
Let $c = \min_{z \in \mathcal{Z}} |u'(z)| > 0$. For any sufficiently small $\varepsilon > 0$, there exists $\delta > 0$ such that if $\|\tilde{u} - u\|_{C^1([0, R])} < \delta$, then $|\tilde{u}'(r)| \ge c/2$ on the $\varepsilon$-neighborhood $\mathcal{Z}^\varepsilon$ of $\mathcal{Z}$, and $\tilde{u}$ possesses exactly one zero in each component of $\mathcal{Z}^\varepsilon$. 
In this singular region $\mathcal{Z}^\varepsilon$, the $L^1$-difference is bounded by the integrability of the singularity, yielding \[
\|f'(\tilde{u}) - f'(u)\|_{L^1(\mathcal{Z}^\varepsilon)} \leq \|f'(\tilde{u})\|_{L^1(\mathcal{Z}^\varepsilon)} +\| f'(u)\|_{L^1(\mathcal{Z}^\varepsilon)} \le C \varepsilon^{q-1}.
\]
Conversely, in the regular region $[0, R] \setminus \mathcal{Z}^\varepsilon$, $u$ and $\tilde u$ are bounded away from zero, ensuring that $f'(\tilde{u}) \to f'(u)$ uniformly. 
Since $\varepsilon > 0$ is arbitrary, the $L^1$-convergence on $[0, R]$ follows immediately.

We now prove that the map $\alpha \mapsto u_\alpha := u(\cdot, \alpha)$ is differentiable with respect to $\alpha$, with derivative $v = \partial_\alpha u_\alpha$. 

Proposition \ref{pro2.5} (iii)--(iv) guarantees that $\{u_{\alpha+h}\}$ is uniformly bounded in $C([0,R])$ for small $h$. Expressing $u_{\alpha+h}$ via the equivalent integral equation
\begin{equation}\label{eq:u-integral}
u_{\alpha+h}(r) = \alpha+h - \int_0^r K(r,s) f(u_{\alpha+h}(s)) \, ds,
\end{equation}
where $K(r,s) = s^{n-1} \int_s^r t^{1-n} \, dt$ and $\partial_r K(r,s) = (s/r)^{n-1}$, we observe that the uniform boundedness of $\partial_r K$ yields a uniform $C^1([0,R])$ bound for $\{u_{\alpha+h}\}$ via the relation $u_{\alpha+h}'(r) = - \int_0^r \partial_r K(r,s) f(u_{\alpha+h}(s)) \, ds$.

By the Arzel\`{a}-Ascoli theorem, any sequence $h_n \to 0$ has a subsequence converging uniformly to some $\bar{u} \in C([0,R])$. Passing to the limit in \eqref{eq:u-integral} identifies $\bar{u} = u_\alpha$ by uniqueness, ensuring $u_{\alpha+h} \to u_\alpha$ in $C([0,R])$. Differentiating \eqref{eq:u-integral} then gives $u_{\alpha+h}'(r) = - \int_0^r \partial_r K(r,s) f(u_{\alpha+h}(s)) \, ds$. The uniform convergence of $u_{\alpha+h}$, the continuity of $f$, and the boundedness of $\partial_r K$ immediately yield the uniform convergence of the derivatives, establishing $u_{\alpha+h} \to u_\alpha$ in $C^1([0,R])$.

For small $h \neq 0$, we define the difference quotient
\[
w_h(r) = \frac{u_{\alpha+h}(r) - u_\alpha(r)}{h}.
\]
This function satisfies the initial value problem
\begin{equation*} 
\bigl(r^{n-1}w_h'(r)\bigr)' + r^{n-1}A_h(r)w_h(r) = 0, \qquad w_h(0) = 1, \quad w_h'(0) = 0,
\end{equation*}
where the averaged coefficient is given by
\[
A_h(r) := \frac{f(u_{\alpha+h}(r)) - f(u_\alpha(r))}{u_{\alpha+h}(r) - u_\alpha(r)} = \int_0^1 f'\bigl((1-t)u_\alpha(r) + t u_{\alpha+h}(r)\bigr) \, dt,
\]
where $A_h$ is defined by the above quotient whenever 
$u_{\alpha+h}(r)\ne u_\alpha(r)$, and by the integral formula
otherwise.
Since $\|u_{\alpha+h} - u_\alpha\|_{C^1([0,R])} \to 0$ as $h \to 0$, the $L^1$-continuity established in \eqref{eq:fprime-continuity} implies that $A_h \to f'(u_\alpha)$ in $L^1([0,R])$. Consequently, the continuous dependence result in Proposition \ref{lem:singular-linear} (ii) guarantees that $w_h \to v$ in $C^1([0,R])$ as $h \to 0$. This convergence proves that $u$ is differentiable with respect to $\alpha$, with $\partial_\alpha u = v$.

Finally, we establish the joint $C^1$-regularity of $u$ with respect to the variables $(r, \alpha)$. The convergence $u_{\alpha+h} \to u_\alpha$ in $C^1([0,R])$ as $h \to 0$ implies that the map $\alpha \mapsto u(\cdot, \alpha)$ is continuous from $\mathbb{R}$ into $C^1([0,R])$. This directly yields the joint continuity of both $u$ and $\partial_r u$ with respect to $(r, \alpha)$. 
Similarly, 
\eqref{eq:fprime-continuity} and Proposition \ref{lem:singular-linear} (ii)
 implies that the map $\alpha \mapsto v= \partial_\alpha u(\cdot, \alpha)$ is continuous into $C^1([0,R])$, ensuring the joint continuity of $\partial_\alpha u$. Therefore, $u(r, \alpha)$ is continuously differentiable   with respect to $(r, \alpha)$.
\end{proof}

\begin{lemma}\label{limitv}
	 Under the assumptions of Lemma \ref{lem-global-v}, let $u$ be a ground state or a bound state solution. If there exists $d \in (0, z_u)$ such that $v(r) \neq 0$ for all $r \in (d, z_u)$, then $v$ is eventually monotone as $r\uparrow z_u$. Furthermore, if $v$ remains bounded as $r\uparrow z_u$, then 
\[
  v'(r) = O\bigl((z_u - r)^{-1}\bigr) \quad \text{as } r\uparrow z_u.
\]
\end{lemma}
\begin{proof}
Since $v$ has a constant sign in $(d, z_u)$ and $u(r) \to 0$ as $r\uparrow z_u$,
conditions (C1') and (C3') imply that $f'(u(r)) < 0$ for $r$
sufficiently close to $z_u$ (using the oddness of $f$ when $u(r)<0$).
 The variational equation 
\[
(r^{n-1}v')' = -r^{n-1}f'(u)v
\]
then implies that $r^{n-1}v'$ is strictly monotone near $z_u$. Consequently, $r^{n-1}v'$ can have at most one zero in this region. Thus, there exists $r_0 \in (d, z_u)$ such that $v'$ has a constant sign on $(r_0, z_u)$, which confirms that $v$ is eventually monotone.

Furthermore, since $v$ is bounded near $z_u$, Proposition \ref{pro2.6} (iv) combined with condition (C3') ensures that there exists a constant $C > 0$ such that
\[
  |(r^{n-1}v')'| = |r^{n-1}f'(u)v| \le C(z_u - r)^{-2} \quad \text{for } r \in (r_0, z_u).
\]
Integrating this inequality from $r_0$ to $r \in (r_0, z_u)$ yields 
\[
|r^{n-1}v'(r)| = O\bigl((z_u - r)^{-1}\bigr) \quad \text{as } r\uparrow z_u.
\]
Since $z_u > 0$, the term $r^{n-1}$ is bounded away from zero near $z_u$. Therefore, we immediately obtain the desired estimate.
\end{proof}
\begin{lemma}\label{lem:tendency-sublinear}
	Suppose that (C1'), (C2), (C3'),  (C4) and (C9) hold with $n\ge2$.
	Let \(u=u(r,\alpha)\) be a ground state or a nodal solution of equation \eqref{eq2.1}.
	Then the following assertions hold.
	
	\begin{enumerate}
		\item[(i)] Let \(r_1\) be the first point such that \(u(r_1)=1\). Then \(v\)
		changes sign in \((0,r_1)\).
		
		\item[(ii)] Suppose that \(u\) is a nodal solution with   \(k\ge2\) simple zeros before \(z_u\). For \(2\le i\le k\), let \(r_i\in(c_{i-1},z_i)\) and
		\(\bar r_{i-1}\in(z_{i-1},c_{i-1})\) be defined by
		\[
		|u(r_i)|=|\;u(\bar r_{i-1})\;|=1 .
		\]
		Then \(v\) changes sign in \((\bar r_{i-1},r_i)\).Moreover, if
	$
	v(c_{i-1})v'(c_{i-1})<0,
	$
	then \(v\) changes sign in \((c_{i-1},r_i)\).
	\end{enumerate}
	
\end{lemma}

\begin{proof}
	Since the analysis is restricted to intervals where $|u| \ge 1$, the non-Lipschitz singularity of $f$ at $u=0$   is entirely avoided. In these regions, by (C9) $f$ is $C^2$, and consequently, the variation $v$ is $C^3$. Therefore, the Wronskian $\varrho$ defined in \cite[Lemma 4.1]{Tang2025} is well-defined and $C^1$. The remainder of the proof proceeds exactly as in \cite[Lemma 4.1]{Tang2025}.
\end{proof}

We shall use the following auxiliary functions as in \cite{Tang2025}:
\begin{align}
	&Q(r)
	:=
	r^n\bigl[u'v'+f(u)v\bigr]
	+(n-2)r^{n-1}u'v, \qquad 
		M(r)
	:=	r^{n-1}(u'v-uv'), \label{eq:def QM}\\
	&T_1(r)
	:=
	Q(r)-g_1(u(r))M(r), \qquad  g_1(u):=	\frac{2f(u)}{uf'(u)-f(u)}.\label{eq:def T1g1}
\end{align}

A direct calculation gives
\begin{equation}\label{AF'}
    Q'(r) = 2r^{n-1}f(u)v, \quad M'(r) = r^{n-1}\bigl[uf'(u)-f(u)\bigr]v, \quad T_1'(r) = -g_1'(u)u'M(r) \quad (u \neq 0).
\end{equation}
Applying the limit theorem for derivatives, we conclude that $Q, M \in C^1([0, z_u))$, while $T_1$ is of class $C^1$ if $u\neq 0$.
 
We also define the “bridging function” as in \cite{Tang2025}, for \(a\in\mathbb R\), wherever \(u'\neq0\),
\begin{equation}\label{eq:Ba-sublinear}
	B_a(r):=
	Q(r)-aM(r)-2F_a(u(r))\frac{r^{n-1}v(r)}{u'(r)},
\end{equation}
where 
\begin{equation} \label{eq:Fa-sublinear}
	F_a(u):=F(u)-\frac a2\bigl[uf(u)-2F(u)\bigr].
\end{equation}
Then
\begin{equation}\label{eq:Baprime-sublinear}
	B_a'(r)=
	-2F_a(u(r))\frac{Q_n(r)}{r u'(r)^2},
\end{equation}
where  for an integral $i$
\[
Q_i(r)=Q(r)+i r^{n-1}u'v .
\]

\begin{lemma}
	\label{lem:no-second-zero-before-z1-sublinear}
	Assume that (C1'), (C2), (C3'), (C4)--(C6), and (C9) hold with $n\geq 2$. In addition, suppose that $u f'(u) - f(u)>0$, $g_1'(u) > 0$ for $u > 0$.
	\begin{enumerate}
		\item[(i)] If \(u\) is a nodal solution with first zero \(z_1<z_u\), then there exists
		\(\tau_1\in(0,r_1)\) such that
		\[
		v(r)>0\quad\hbox{in }(0,\tau_1),\qquad
		v(\tau_1)=0,\qquad
		v(r)<0\quad\hbox{in }(\tau_1,z_1].
		\]
		Moreover,
		\begin{equation*}
			Q(r),\,M(r)>0,\quad r\in(0,z_1];\qquad
			T_1(r),\,T_1'(r)>0,\quad r\in(0,z_1).
		\end{equation*}
		
		\item[(ii)] If \(u\) is a ground state with support radius \(z_u\), then there exists
		\(\tau_1\in(0,r_1)\) such that
		\[
		v(r)>0 \hbox{ in }(0,\tau_1),\qquad
		v(\tau_1)=0,\quad
		v(r)<0 \hbox{ in }(\tau_1,z_u),\quad \lim_{r\uparrow z_u}v(r)=-\infty.
		\]
		In addition,
		\[
		Q(r),\,M(r)>0,\qquad T_1(r),\,T_1'(r)>0,
		\qquad r\in(0,z_u).
		\]
	\end{enumerate}
\end{lemma}

\begin{proof}
	We set
	\[
	R_*=
	\begin{cases}
		z_1, & \hbox{if }u\hbox{ is a nodal solution},\\
		z_u, & \hbox{if }u\hbox{ is a ground state}.
	\end{cases}
	\]
	 By Lemma \ref{lem-global-v} and \ref{lem:tendency-sublinear},
	there exists \(\tau_1\in(0,r_1)\) such that
	\[
	v>0\quad\hbox{in }(0,\tau_1),
	\qquad
	v(\tau_1)=0,
	\qquad
	v'(\tau_1)<0.
	\]
 
Recall \eqref{eq:def QM}, \eqref{eq:def T1g1}, and \eqref{AF'} that
\[
Q'(r)=2r^{n-1}f(u)v,
\qquad
M'(r)=r^{n-1}\bigl[uf'(u)-f(u)\bigr]v,
\]
and
\[
T_1'(r)=-g_1'(u)u'M(r),
\qquad
g_1(u)=\frac{2f(u)}{uf'(u)-f(u)} .
\]

On \((0,\tau_1)\), we have
\[
u>1,\qquad u'<0,\qquad f(u)>0,\qquad
uf'(u)-f(u)>0,\qquad v>0.
\]
Therefore
$
Q'(r)>0,\ M'(r)>0,\ r\in(0,\tau_1).
$
Since \(Q(0)=M(0)=T_1(0)=0\), it follows that
\[
Q(r)>0,\qquad M(r)>0,\qquad r\in(0,\tau_1].
\]
Using \(g_1'(u)>0\), \(u'<0\), and \(M>0\), we also get
\[
T_1'(r)>0,\qquad T_1(r)>0,\qquad r\in(0,\tau_1].
\]

We first show that if $M>0$ in $(0,\tau)$ for some $\tau\in(\tau_1,R_*)$, then 
$v<0$ in $(\tau_1,\tau]$. Otherwise,  if   \(v\) had a first
zero \(\sigma\in(\tau_1,\tau]\), then \(v'(\sigma)>0\) by uniqueness, and hence
\[
 M(\sigma)
=
\sigma^{n-1}\bigl(u'(\sigma)v(\sigma)-u(\sigma)v'(\sigma)\bigr)
=
-\sigma^{n-1}u(\sigma)v'(\sigma)<0,
\]
because \(u(\sigma)>0\). This contradicts \(M(\sigma)\geq 0\).

We now prove that \(M>0\) in $(\tau_1,R_*)$ and hence $v<0$ in this interval.
 Suppose, to the contrary,
that \(M\) has a first zero \(\tau_{1m}\in(\tau_1,R_*)\), namely
\[
M(r)>0\quad\hbox{in }(0,\tau_{1m}),
\qquad
M(\tau_{1m})=0 .
\]
Then \(v<0\) in \((\tau_1,\tau_{1m}]\), and  $T_1'(r)=-g_1'(u)u'M(r)>0$ in $(0,\tau_{1m})$, combining which with \eqref{eq:def T1g1} and $T_1(\tau_1)>0$, we obtain that
\begin{equation}\label{eq3.16}
	Q(\tau_{1m})=T_1(\tau_{1m})+g_1(u(\tau_{1m}))M(\tau_{1m})=T_1(\tau_{1m})>0.
\end{equation} 
At \(r=\tau_{1m}\), the equality \(M(\tau_{1m})=0\) gives
$
u'v=uv' .
$
Since \(u(\tau_{1m})>0\), \(u'(\tau_{1m})<0\), and
\(v(\tau_{1m})<0\), we obtain
\[
v'(\tau_{1m})>0,
\qquad
\frac{u(\tau_{1m})}{v(\tau_{1m})}<0 .
\]
Using \(u'v=uv'\) at $\tau_{1m}$,  \eqref{P1}, and \eqref{eq3.16}, we get
\[
\begin{aligned}
	P_1(\tau_{1m})
	&=
	\tau_{1m}^{n}
	\bigl[u'(\tau_{1m})^2
	+u(\tau_{1m})f(u(\tau_{1m}))\bigr]
	+(n-2)\tau_{1m}^{n-1}u(\tau_{1m})u'(\tau_{1m})        \\
	&=
	\Big(
	\tau_{1m}^{n}
	\bigl[u'(\tau_{1m})v'(\tau_{1m})
	+f(u(\tau_{1m}))v(\tau_{1m})\bigr]
	+(n-2)\tau_{1m}^{n-1}u'(\tau_{1m})v(\tau_{1m})
	\Big)
	\frac{u(\tau_{1m})}{v(\tau_{1m})}                     \\
	&=
	Q(\tau_{1m})\frac{u(\tau_{1m})}{v(\tau_{1m})}<0.
\end{aligned}
\]
This contradicts Proposition \ref{pro2.9} and Proposition \ref{pro2.9-planar}, which give \(P_1>0\). Therefore
\[
M(r)>0,\qquad r\in(0,R_*),
\]
and  
\[
v(r)<0,\qquad r\in(\tau_1,R_*).
\]

Since \(M>0\) in \((0,R_*)\), we get by $T_1'(r)=-g_1'(u)u'M(r)$ that
\[
T_1'(r)>0,\qquad T_1(r)>0,\qquad r\in(0,R_*).
\]
At \(r=r_1\), \(u(r_1)=1\), and hence \(g_1(u(r_1))=g_1(1)=0\). Therefore
\[
Q(r_1)=T_1(r_1)+g_1(u(r_1))M(r_1)=T_1(r_1)>0.
\]
Moreover, by $Q'(r)=2r^{n-1}f(u)v$, we know that \(Q\) increases on \((0,\tau_1)\), decreases on \((\tau_1,r_1)\), but
remains positive at \(r_1\). Since \(f(u)<0\) and \(v<0\) on \((r_1,R_*)\), we have
\begin{equation*} 
Q'(r)=2r^{n-1}f(u)v>0,\qquad r\in(r_1,R_*).
\end{equation*}
Consequently,
\[
Q(r)>0,\qquad r\in(0,R_*).
\]
If \(u\) is a nodal solution, then \(R_*=z_1\). We have already proved
\[
v<0,\qquad Q>0,\qquad M>0
\quad\hbox{in }(0,z_1).
\]  Since  
\(z_1\) is simple, so \(u'(z_1)<0\), and \(v,v'\) are finite at \(z_1\). If
\(v(z_1)=0\), then, because \(v<0\) in \((\tau_1,z_1)\), we would have
\(v'(z_1) >0\). Hence
\[
Q(z_1)
=
z_1^n u'(z_1)v'(z_1)<0,
\]
contradicting the positivity of \(\lim_{r\uparrow z_1}Q(r)\). Therefore
\(v(z_1)<0\). It follows that
\[
M(z_1)=z_1^{n-1}u'(z_1)v(z_1)>0,
\]
and also \(Q(z_1)>0\). This proves part (i).

To complete the proof of (ii),
it remains to prove that $\lim_{r\uparrow z_u} v(r) = -\infty$ when $u$ is the ground state. Since $v$ is monotone by Lemma \ref{limitv}, it suffices to show that $v$ is unbounded. Assume to the contrary that $v$ is bounded near $z_u$. 
By Lemma \ref{limitv}, $v'(r)=O((z_u-r)^{-1})$.
Combining this with Proposition \ref{pro2.6} (iv) yields 
\[
u'v' = O\left((z_u - r)^{\frac{2q-2}{2-q}}\right) \to 0 \quad \text{as } r\uparrow z_u.
\]
Moreover, $u'v\to0$ and $f(u)v\to0$.
Hence, by the definition of $Q$, we have $Q(r)\to0$.
This contradicts $Q(r)\ge Q(r_1)>0$ ...
which implies $Q(r) \to 0$. This directly contradicts the property that $Q > 0$ and is strictly increasing on $(r_1, z_u)$. Hence, $v$ must diverge to $-\infty$ as $r\uparrow z_u$.

\end{proof}

\begin{remark}
For the nonlinearity $f(u) = u - |u|^{q-2}u$ with $1 < q < 2$, a direct computation yields
\[
u f'(u) - f(u) = (2-q) u |u|^{q-2} \quad \text{and} \quad g_1(u) = \frac{2}{2-q} \bigl(|u|^{2-q} - 1\bigr).
\]
Consequently, the following
assumptions in Lemma \ref{lem:no-second-zero-before-z1-sublinear} are   satisfied:
\[
u f'(u) - f(u) = (2-q) u^{q-1} > 0 \quad \text{and} \quad g_1'(u) = 2 u^{1-q} > 0\text{ for }  u > 0.
\] 
\end{remark}

\begin{lemma}	\label{lem:3.6}
	Under the assumption of Lemma \ref{lem:no-second-zero-before-z1-sublinear},
	let \(u\) be either a
	nodal solution with at least two zeros before \(z_u\), or a bound state. Let
	\(z_1\) be the first zero of \(u\), and let \(c_1\) be the unique critical point
	of \(u\) after \(z_1\) in Phase \(1\). Then there exists \(\tau_1\in(0,r_1)\)
	such that
	\[
	v>0\quad\hbox{in }(0,\tau_1),\qquad
	v(\tau_1)=0,\qquad
	v<0\quad\hbox{in }(\tau_1,c_1].
	\]
	Moreover,
	\[
	Q(r)>0\quad\hbox{for }r\in(0,z_1]\cup[\bar b_1,c_1],\qquad{and}\quad 
	M(r)>0,\quad Q_1(r)>0,\quad Q_2(r)>0
	\quad\hbox{for }r\in(0,c_1].
	\]
 
\end{lemma}

\begin{proof}
The proof follows the same argument as in \cite[Lemma 4.3]{Tang2025}. The analysis is restricted to the interval $[0, c_1]$, which contains exactly one zero of $u$ (namely, its first simple zero $z_1$). At $z_1$, the solution $u$ is of class $C^2$, while the variation $v$ is well-defined and belongs to $C^1([0, c_1]) \cap C^2([0, c_1] \setminus \{z_1\})$. Consequently, the auxiliary functions $Q_i$ ($i=0,1,2,n$) and $M$ employed in the original proof are of class $C^1([0, c_1])$, with $B_0$ being $C^1$ on $(0, c_1)$. This regularity, together with the other properties utilized in \cite{Tang2025}, is entirely sufficient to carry out the argument, ensuring that the modified conditions (C1') and (C3') have no impact on this step.
\end{proof}

\section{Transition to later phases}

In this section, we transition from Phase 1 to the subsequent phases. We restrict our analysis to equation \eqref{eq2.1} with
\[
f(u) = u - |u|^{q-2}u, \quad F(u) = \frac{1}{2}u^2 - \frac{1}{q}|u|^q, \qquad 1 < q < 2.
\]
In contrast to \cite[Section 5]{Tang2025}, we assume $n \ge 2$. Throughout this section, we shall repeatedly use the auxiliary functions $Q_i$, $M$, and $P$ introduced in the previous sections.
\subsection{The phase transition lemma}

Our main auxiliary function in the later phases is
\begin{equation}\label{eq:def-T2-g2}
	T_2(r):=Q(r)-g_2(u(r))M(r),
	\qquad
	g_2(u):=\frac{2F(u)}{uf(u)-2F(u)},\quad u\neq 0 .
\end{equation}
For the sublinear nonlinearity, since
\[
uf(u)-2F(u)=\frac{2-q}{q}|u|^q,
\]
we may write, for \(u\neq0\),
\begin{equation}\label{eq:g2-explicit-sublinear} 
	g_2(u)=\frac{q|u|^{2-q}-2}{2-q}.
\end{equation}
In particular, \(g_2\) has a finite continuous extension at \(u=0\), namely
\(g_2(0)=-2/(2-q)\).

Recall also that
\[
T_1(r):=Q(r)-g_1(u(r))M(r),
\qquad
g_1(u):=\frac{2f(u)}{uf'(u)-f(u)} .
\]
Since
\[
uf'(u)-f(u)=(2-q)u|u|^{q-2},
\]
we have
\[
g_1(u)=\frac{2}{2-q}\bigl(|u|^{2-q}-1\bigr).
\]
Consequently,
\begin{equation}\label{eq:g1-g2-relation-sublinear}
	g_1(u)=g_2(u)+|u|^{2-q},
\end{equation}
and hence
\begin{equation}\label{eq:T2-T1-relation-sublinear}
	T_2(r)=T_1(r)+|u(r)|^{2-q}M(r).
\end{equation}

We shall also use the following connection identity. Let
\[
\omega(r):=-\frac{ru'(r)}{u(r)}
\]
whenever \(u(r)\neq0\), and define
\begin{equation}\label{eq:theta-sublinear}
	\theta(r):=\frac{r^{n-1}v}{u'}[uf(u)-2F(u)]=\frac{2-q}{q}\,
	\frac{r^{n-1}v(r)}{u'(r)}\,|u(r)|^q
\end{equation}
whenever \(u'(r)\neq0\). Then, on every interval where \(u u'\neq0\),
\begin{equation}\label{eq:connection-identity-sublinear}
	Q(r)-P(r)\frac{v(r)}{u(r)}
	=
	\omega(r)\bigl[M(r)-\theta(r)\bigr].
\end{equation}
Indeed, using
\[
M=r^{n-1}(u'v-uv')
\]
and the definitions of \(P\) and \(Q\), we compute
\[
\begin{aligned}
	Q-P\frac vu
	&=
	r^n\left(u'v'+f(u)v-\frac vu\bigl(u'^2+2F(u)\bigr)\right)  \\
	&=
	r^n\left(u'v'-\frac{u'^2v}{u}\right)
	+r^n\frac{uf(u)-2F(u)}{u}v                                  \\
	&=
	\omega M
	-\omega\left(\frac{r^{n-1}v}{u'}[uf(u)-2F(u)]\right)          \\
	&=
	\omega(M-\theta).
\end{aligned}
 \]

A direct differentiation of \(T_2\) gives
\begin{equation}\label{eq:T2prime-first-sublinear}
	T_2'(r)
	=
	(2-q)r^{n-1}u(r)v(r)
	-q\,\frac{u(r)u'(r)}{|u(r)|^q}M(r),
\end{equation}
or, equivalently,
\begin{equation}\label{eq:T2prime-second-sublinear}
	T_2'(r)
	=
	-q\,\frac{u(r)u'(r)}{|u(r)|^q}
	\bigl[M(r)-\theta(r)\bigr],
\end{equation}
on intervals where \(u\neq0\). To verify \eqref{eq:T2prime-first-sublinear}, note that
\[
T_2'=Q'-g_2'(u)u'M-g_2(u)M',
\]
where
\[
Q'=2r^{n-1}f(u)v,\qquad
M'=r^{n-1}\bigl[uf'(u)-f(u)\bigr]v.
\]
Moreover,
\[
g_2'(u)=q\,u|u|^{-q},
\]
and
\[
\begin{aligned}
	2f(u)-g_2(u)\bigl[uf'(u)-f(u)\bigr]
	&=
	2\bigl(u-u|u|^{q-2}\bigr)
	-\bigl(q|u|^{2-q}-2\bigr)u|u|^{q-2} \\
	&=
	(2-q)u .
\end{aligned}
\]
This proves \eqref{eq:T2prime-first-sublinear}; \eqref{eq:T2prime-second-sublinear}
then follows from the definition of \(\theta\).

We now state the phase transition lemma, whose proof will be completed after the
auxiliary results in this section.

\begin{lemma}[Phase transition lemma]\label{lem:phase-transition-sublinear}
	Let \(n\ge2\), and let \(u\) be a nodal solution of \eqref{eq2.1}. Suppose that \(u\) has exactly \(k\)
	zeros before \(z_u\) if \(u\) is a bound state, and has exactly \(k+1\) zeros before
	\(z_u=+\infty\) if \(u\) is not a bound state. Then
	\[
	Q(c_1)>0,\qquad M(c_1)>0,\qquad T_2(c_1)>0.
	\]
	Furthermore, the following statements hold.
	
	\begin{enumerate}
		\item[(i)] If \(i\in\{2,\ldots,k\}\) and
		\[
		Q(c_{i-1})>0,\qquad M(c_{i-1})>0,\qquad T_2(c_{i-1})>0,
		\]
		then \(v\) has a unique zero \(\tau_i\) in \([c_{i-1},c_i]\), with
		\[
		\tau_i\in(c_{i-1},r_i),
		\]
		and
		\[
		Q(c_i)>0,\qquad M(c_i)>0,\qquad T_2(c_i)>0.
		\]
		
		\item[(ii)] If \(u\) is not a bound state and
		\[
		Q(c_k)>0,\qquad M(c_k)>0,\qquad T_2(c_k)>0,
		\]
		then \(v\) has a unique zero \(\tau_{k+1}\) in \([c_k,z_{k+1}]\), with
		\[
		\tau_{k+1}\in(c_k,r_{k+1}).
		\]
		
		\item[(iii)] If \(u\) is a bound state and
		\[
		Q(c_k)>0,\qquad M(c_k)>0,\qquad T_2(c_k)>0,
		\]
		then \(v\) has a unique zero \(\tau_{k+1}\) in \([c_k,z_u)\), with
		$
		\tau_{k+1}\in(c_k,r_{k+1}).
		$
		Moreover, \(v\) is eventually monotone as \(r\uparrow z_u\), and 
		$
		\lim_{r\uparrow z_u}|v(r)|=\infty .
		$
	\end{enumerate}
\end{lemma}

\subsection{Locating the zeros of \(v(r)\) and keeping a record of signs}

The first step toward the phase transition lemma is to locate the first zero of
\(v\) in each later phase and to record the signs of all quantities needed later.
If $u$ is a bound state, we denote $z_u=z_{k+1}$.

\begin{lemma}\label{lem:locating-v-zero-sublinear}
	Let the assumptions in Lemma \ref{lem:no-second-zero-before-z1-sublinear} hold.
	Suppose that \(n\ge2\) and
	$
	Q(c_{i-1})>0,\ M(c_{i-1})>0
	$
	for some \(i\in\{2,\ldots,k,k+1\}\). Then \(v\) changes sign in
	\((c_{i-1},r_i)\). Let \(\tau_i\) be the first zero of \(v\) in
	\((c_{i-1},r_i)\). Then, for all $r\in(c_{i-1},\tau_i)$,
	\begin{equation}\label{eq:sign-record-later-phase}
		uf(u)>0,\qquad f'(u)>0,\qquad uv>0,\qquad u'v'>0,
	\end{equation}
	and
	\begin{equation}\label{eq:sign-record-later-phase-2}
		uu'<0,\qquad uv'<0,\qquad u'v<0,\qquad vv'<0.
	\end{equation}	
	In addition, \eqref{eq:sign-record-later-phase}--\eqref{eq:sign-record-later-phase-2}
	  hold for \(i=1\), with \(c_0=0\) and \(\tau_1\) being the unique zero of \(v\)
	in Phase \(1\).
\end{lemma}
\begin{proof}
	The proof is the same as that of \cite[Lemma 5.2]{Tang2025}. It depends exclusively on the abstract conditions (C1'), (C2), (C3'), (C4), and (C9), making it independent of the specific formula for $f$. Additionally, the interval under consideration does not contain any zeros of $u$.
\end{proof}

\subsection{Positivity of an integral relative to \(Q_n\)}

For \(i\ge 1\), recall that
\[
Q_n(r)=Q(r)+n r^{n-1}u'(r)v(r).
\]

We introduce an auxiliary functional
	\[
	A(r):=-\frac{r^{n-1}v(r)}{u'(r)}.
	\]
 
	When $vu'\neq 0$, taking the logarithmic derivative  yields
\[
	\frac{A'}{A} = (\log |A|)' = \frac{n-1}{r} + \frac{v'}{v} - \frac{u''}{u'}.
\]
Using the governing equation $u'' = -\frac{n-1}{r}u' - f(u)$, we substitute $\frac{u''}{u'} = -\frac{n-1}{r} - \frac{f(u)}{u'}$ to obtain
\begin{equation}\label{eq4.13}
	\frac{A'}{A} = \frac{2(n-1) -S(r)}{r},\quad S(r):=- r\frac{v'}{v} - r\frac{f(u)}{u'}.
\end{equation}
By substituting $u'v' + f(u)v = -\frac{S(r)}{r} u'v$ into the definition of $Q_i(r) = r^n[u'v' + f(u)v] + (n-2+i)r^{n-1}u'v$, we obtain
\begin{equation}\label{Qi_S}
Q_i(r) = r^{n-1} u'v \Big[ (n - 2 + i) - S(r) \Big].
\end{equation}
 Therefore,
\begin{equation*} 
	 A'
	=
	-\frac{Q_n}{r u'^2}.
	\end{equation*}

\begin{lemma}
	\label{lem:residual-integral-alln-sublinear}
Under the assumptions of Lemma \ref{lem:locating-v-zero-sublinear}, let $\tau_i \in (c_{i-1}, r_i)$ be the first zero of $v$. If $\tau_i > b_i$, then for every $\widetilde{u} \ge |u(b_i)| = \alpha_*$, one has
\begin{equation}\label{eq:unified_integral}
  B_a(\tau_i)-B_a(b_i)=  \int_{b_i}^{\tau_i} -2F_{a}(u(r)) \frac{Q_n(r)}{r u'(r)^2} \, dr > 0,
\end{equation}
where $a=g_2(\widetilde u)\geq 0$.  
 
\end{lemma}

\begin{proof}
Without loss of generality, we assume $i$ is an odd number. The case for even $i$ follows identically by the symmetry transformation $(u, v) \mapsto (-u, -v)$, which leaves the governing equations, the quantities $Q, M, Q_n$, and the integral invariant.

For an odd $i$, on the interval $[b_i, \tau_i)$, we have $u > 0$, $u' < 0$, $v > 0$, and $u(r) \in (1, \alpha_*)$. 
By the definition of $F_a(u)$,  
  the integral in \eqref{eq:unified_integral} splits into two parts:
\[
    \int_{b_i}^{\tau_i} \big( -2F_a(u) \big) \frac{Q_n}{r u'^2} \, dr = \underbrace{\int_{b_i}^{\tau_i} (-2F(u)) \frac{Q_n}{r u'^2} \, dr}_{I_1} + a \underbrace{\int_{b_i}^{\tau_i} h_1(u) \frac{Q_n}{r u'^2} \, dr}_{I_2}.
\]

\textbf{Step 1: Proof of $I_1 > 0$.}
Recall the bridging function $B_0(r) = Q(r) - 2F(u)\frac{r^{n-1}v}{u'}$ with
\(
    B_0'(r) = -2F(u) \frac{Q_n(r)}{r u'(r)^2}.
\)
Therefore,  
\[
    I_1 = B_0(\tau_i) - B_0(b_i).
\]
Since $F(u(b_i)) = F(\alpha_*) = 0$ and $v(\tau_i) = 0$, we have $B_0(b_i) = Q(b_i)$ and $B_0(\tau_i) = Q(\tau_i)$. 
On the interval $(b_i, \tau_i)$, we know $f(u) > 0$ and $v > 0$. From the identity $Q'(r) = 2r^{n-1}f(u)v$, it follows that $Q'(r) > 0$. Consequently, $Q(\tau_i) > Q(b_i)$, which yields $I_1 > 0$.

\textbf{Step 2: Proof of $I_2 > 0$.}
From the identity $Q_2(r) = r^{n-1}(-u')v\big(S(r) - n\big)$ in  \eqref{Qi_S} and the fact that $Q_2'(r) = 2r^{n-1}u'v' > 0$ on $(c_{i-1}, \tau_i)$, we know $Q_2(r)$ is strictly increasing. Since $u'(c_{i-1})=0$, we have $Q_2(c_{i-1}) = Q(c_{i-1}) > 0$. Consequently, $Q_2(r) > 0$ for all $r \in (c_{i-1}, \tau_i)$. Given that $-u' > 0$ and $v > 0$ on $(b_i, \tau_i)$, the identity immediately implies:
	\[
	S(r)>n,\quad r\in(b_i,\tau_i).
	\]
 
	Recall $$\omega=-\frac{ru'(r)}{u(r)}.$$ 
	By Proposition \ref{pro2.9} (i) (ii) if \(n\ge3\), and by Proposition
	\ref{pro2.9-planar} (i) (ii) if \(n=2\), $\omega'(r) > 0$ as long as $u(r) \neq 0$. Thus, $\omega(r)$ is strictly increasing, giving $\omega(r) > \omega(b_i)$ for $r > b_i$. Furthermore, evaluating the Pohozaev function at $b_i$ where $F(u(b_i))=0$ yields $P(b_i) = b_i^{n-1}|u(b_i)u'(b_i)|\big(\omega(b_i) - (n-2)\big) > 0$, which implies $\omega(b_i) > n-2$.  
	Therefore
	\[
	\omega(r)>\omega(b_i)>n-2
	\qquad\hbox{for }r\in(b_i,\tau_i).
	\]

	 By \eqref{eq4.13}, we have
	\[
	\frac{d}{dr}\log(Au)
	=
	\frac{A'}A+\frac{u'}u
	=
	\frac{2(n-1)-S-\omega}{r}.
	\]
	Since \(S>n\) and \(\omega>n-2\), the numerator is negative. Hence,
	\(Au\) is strictly decreasing, and
	\begin{equation}\label{eq4.14}
		A(r)u(r)<A(b_i)u(b_i)=A(b_i)\alpha_*,
		\quad r\in(b_i,\tau_i).
	\end{equation}

Now, we apply integration by parts to $I_2$,
\[
    \begin{aligned}
        I_2 &= -\int_{b_i}^{\tau_i} h_1(u) A'(r) \, dr \\
        &= h_1(u(b_i))A(b_i) - h_1(u(\tau_i))A(\tau_i) + \int_{b_i}^{\tau_i} h_1'(u) u' A \, dr\\
		&= h_1(\alpha_*)A(b_i) - \int_{b_i}^{\tau_i}h_1'(u) (-u') A \, dr,
    \end{aligned}
\]
where we used  $v(\tau_i) = 0$,   $A(\tau_i) = 0$, and $u(b_i) = \alpha_*$.  
 Applying the bound $A(r) < A(b_i)\frac{\alpha_*}{u(r)}$ in \eqref{eq4.14} and the fact that $h_1'(u)=uf'(u)-f(u) > 0$ for $u\geq 1$ by (C4) and (C9), we obtain,
\begin{equation}\label{I2>}
    I_2 > A(b_i) \left[ h_1(\alpha_*) - \alpha_* \int_{b_i}^{\tau_i} \frac{h_1'(u)}{u} (-u') \, dr \right].
\end{equation}
We evaluate the integral via the substitution $s = u(r)$,
\[
    \int_{b_i}^{\tau_i} \frac{h_1'(u)}{u} (-u') \, dr = \int_{u(\tau_i)}^{\alpha_*} \frac{h_1'(s)}{s} \, ds.
\]
Using the definition $h_1(s) = sf(s) - 2F(s)$, we find $h_1'(s) = sf'(s) - f(s)$, which gives $\frac{h_1'(s)}{s} = f'(s) - \frac{f(s)}{s}$. Integrating this yields:
\[
    \int_{u(\tau_i)}^{\alpha_*} \frac{h_1'(s)}{s} \, ds = \int_{u(\tau_i)}^{\alpha_*} f'(s) \, ds - \int_{u(\tau_i)}^{\alpha_*} \frac{f(s)}{s} \, ds = f(\alpha_*) - f(u(\tau_i)) - \int_{u(\tau_i)}^{\alpha_*} \frac{f(s)}{s} \, ds.
\]
Substituting this back into \eqref{I2>}, and using the fact that $h_1(\alpha_*) = \alpha_* f(\alpha_*) - 2F(\alpha_*) = \alpha_* f(\alpha_*)$, we get,
\[ 
      I_2 >   \alpha_* \left( f(u(\tau_i)) + \int_{u(\tau_i)}^{\alpha_*} \frac{f(s)}{s} \, ds \right)>0,
\]
because $u(\tau_i) > 1$ and $f(s) > 0$ for all $s \in (1, \alpha_*)$. Since $A(b_i) > 0$, we conclude that $I_2 > 0$.
\end{proof}

\begin{remark}
	Lemma \ref{lem:residual-integral-alln-sublinear} is independent of the specific form of $f$, and it generalizes the result of \cite[Lemma 5.4]{Tang2025} to $f(u) = |u|^{p-1}u - u$ for $p \in (1, \frac{n+2}{n-2})$ when $n \ge 3$, and for $p > 1$ when $n = 2$.
\end{remark}

\subsection{Positivity of $T_2$}

\begin{lemma}
	\label{lem:Qn-positive-direct-sublinear}
	Let $n\geq 2$, \(i\in\{2,\dots,k\}\). Assume that
\begin{equation}\label{eq:4.17}
		Q(c_{i-1})>0,\quad M(c_{i-1})>0,\quad
		Q_1(r)>0\quad\hbox{for }r\in[c_{i-1},c_i].
	\end{equation}
	Then \(v\) has a unique zero \(\tau_i\) in \([c_{i-1},c_i]\), and
	$
	\tau_i\in(c_{i-1},r_i).
	$	
\end{lemma}

\begin{proof}
By Lemma \ref{lem:locating-v-zero-sublinear}, \(v\) changes sign in
\((c_{i-1},r_i)\). Let \(\tau_i\) be the first zero. Then
$
u'(\tau_i)v'(\tau_i)>0.
$
If \(v\) had a second zero \(\widetilde\tau_i\in(\tau_i,c_i]\), then either
$
Q_1(\widetilde\tau_i)
=
\widetilde\tau_i^n u'(\widetilde\tau_i)v'(\widetilde\tau_i)<0,
$
if \(\widetilde\tau_i<c_i\), or \(Q_1(c_i)=0\), if \(\widetilde\tau_i=c_i\).
Both contradict \(Q_1>0\) on \([c_{i-1},c_i]\). Hence \(v\) has a unique zero in
Phase \(i\).	
\end{proof}

\begin{lemma}\label{lem:comparison-sublinear}
Let $n\geq 2$, \(i\in\{2,\dots,k\}\).	Assume   \eqref{eq:4.17}
	and  
	\begin{equation}\label{eq:4.18}
		Q(r)>0\quad\hbox{for }r\in(c_{i-1},b_i]\cup[\bar b_i,c_i],\qquad 	Q(\bar b_i)>Q(b_i).
	\end{equation}
	For \(\mu\in[\alpha_*,|u(c_i)|)\), let
	\(r_{i\mu}\in(c_{i-1},b_i]\) and
	\(\bar r_{i\mu}\in[\bar b_i,c_i)\) be determined by
	$
	|u(r_{i\mu})|=|u(\bar r_{i\mu})|=\mu.
	$
	Then for every \(\mu\in[\alpha_*,|u(c_i)|)\),
	\begin{equation}\label{phir}
	\phi(\bar r_{i\mu})>\phi(r_{i\mu}),
	\quad \text{where \ }
		\phi(r):=\frac{Q(r)}{r^{n-1}|u'(r)|}.
	\end{equation}
	In particular, both \eqref{eq:4.18} and \eqref{phir} hold for $i=1$.
\end{lemma}
\begin{proof}
	The proof is identical to that of \cite[Lemma 5.5]{Tang2025}. For $i=1$, it relies only on the assumptions in Lemma \ref{lem:3.6} and the monotonicity of $B_0$; for $i \ge 2$, it depends solely on the abstract assumptions (C1') and (C2) on $f$. In either case, the argument is independent of the precise form of $f(u)$. Furthermore, since the proof is restricted to intervals where $u \neq 0$, the non-Lipschitz behavior of $f$ at $0$ does not affect the argument.
\end{proof}

\begin{lemma} 
	\label{lem:T2-positive-sublinear}
	Let \(n\ge2\), and   \(i\in\{2,\dots,k\}\). Assume \eqref{eq:4.17},
	\eqref{eq:4.18},
	and
	\begin{equation}\label{eq:4.20}
		T_2(r)>0\quad\hbox{for }r\in[c_{i-1},b_i];
		\qquad
		Q_2(r)>0\quad\hbox{for }r\in[c_{i-1},c_i].
	\end{equation}
	Then
	$
	T_2(r)>0,\  r\in[\bar b_i,c_i].
	$ In particular, $T_2(r)>0$ on $[\bar b_1,c_1].$
\end{lemma}

\begin{proof}
By \eqref{eq:T2-T1-relation-sublinear} and Lemma \ref{lem:no-second-zero-before-z1-sublinear}, $T_2(r) > 0$ for $r \in (0, z_1)$. Combined with Lemmas \ref{lem:comparison-sublinear} and \ref{lem:3.6}, we observe that the assumptions hold for $i=1$, except that $M, Q, Q_1, Q_2$, and $T_2$ vanish at $c_0 = 0$. However, as discussed in the proof of \cite[Lemma 5.6]{Tang2025}, the argument for $i=1$ follows the exact same lines as that for $2 \le i \le k$.

Let $2\le i\le k$.	By Lemma \ref{lem:Qn-positive-direct-sublinear},  \(v\) has a unique zero
	\(\tau_i\in(c_{i-1},r_i)\) in Phase \(i\). Hence \(u'v>0\) on
	\((\tau_i,c_i]\). In particular,
	\begin{equation}\label{eq:Qn-positive-right-part}
		Q_n(r)=Q_1(r)+(n-1)r^{n-1}u'v>0,
		\qquad r\in[\bar b_i,c_i].
	\end{equation}
	At \(b_i\) and \(\bar b_i\), we have \(F(u)=0\), hence \(g_2(u)=0\). Therefore, recalling \eqref{eq:def-T2-g2}, we obtain that
	\[
	T_2(b_i)=Q(b_i)>0,
	\qquad
	T_2(\bar b_i)=Q(\bar b_i)>Q(b_i)>0.
	\]
	Moreover, \eqref{eq:T2prime-first-sublinear} gives
	\begin{equation}\label{eq:T2prime-used}
		T_2'(r)
		=
		(2-q)r^{n-1}u(r)v(r)
		-
		q\,\frac{u(r)u'(r)}{|u(r)|^q}M(r).
	\end{equation}
	Since \(u'(c_i)=0\) and \(u(c_i)v(c_i)>0\), we get
	\[
	T_2'(c_i)=(2-q)c_i^{n-1}u(c_i)v(c_i)>0.
	\]
	If \(T_2\) has no critical point in \((\bar b_i,c_i)\), then \(T_2'>0\) on this
	interval, and the conclusion follows from \(T_2(\bar b_i)>0\).
	
	Assume now that \(T_2\) has a critical point in \((\bar b_i,c_i)\). We first prove
	that such a critical point is unique and is a strict local minimum. Recall 
	\[
	B_0(r):=
	Q(r)-2F(u(r))\frac{r^{n-1}v(r)}{u'(r)},\quad
	B_0'(r)
	=
	-2F(u(r))\frac{Q_n(r)}{r u'(r)^2},
	\]
	and that,
	\[
	T_2'(r)
	=
	-q\frac{u(r)u'(r)}{|u(r)|^q}\bigl(M(r)-\theta(r)\bigr),
	\]
	with \[
	\theta(r)=\frac{2-q}{q}\frac{r^{n-1}v(r)}{u'(r)}|u(r)|^q,
	\]
	it  follows from
	\(T_2=Q-g_2M\), \(g_2(u)=2F(u)/(uf(u)-2F(u))\), and
	\(uf(u)-2F(u)=\frac{2-q}{q}|u|^q\) 
	that
	\begin{equation}\label{eq:T2-B0-identity-sublinear}
		T_2(r)
		=
		B_0(r)
		+
		\frac{2F(u(r))}{(2-q)u(r)u'(r)}\,T_2'(r),
		\qquad u(r)u'(r)\neq0 .
	\end{equation}
	
	Let \(t_i\in(\bar b_i,c_i)\) be a critical point of \(T_2\). Differentiating
	\eqref{eq:T2-B0-identity-sublinear} at \(t_i\), using \(T_2'(t_i)=0\), gives
	\[
	0
	=
	-2F(u(t_i))\frac{Q_n(t_i)}{t_i u'(t_i)^2}
	+
	\frac{2F(u(t_i))}{(2-q)u(t_i)u'(t_i)}\,T_2''(t_i).
	\]
	Since \(F(u(t_i))>0\) on \((\bar b_i,c_i)\), we obtain
	\begin{equation*} 
		T_2''(t_i)
		=
		(2-q)u(t_i)u'(t_i)\frac{Q_n(t_i)}{t_i u'(t_i)^2}.
	\end{equation*}
	On \((\bar b_i,c_i)\), one has \(u u'>0\), and by
	\eqref{eq:Qn-positive-right-part}, \(Q_n>0\). Hence \(T_2''(t_i)>0\). Thus every
	critical point of \(T_2\) in \((\bar b_i,c_i)\) is a strict local minimum. Hence
	there can be at most one such critical point. Denote it by \(t_i\). It remains to
	prove \(T_2(t_i)>0\).
	
	Choose the unique point \(s_i\in(c_{i-1},b_i)\) such that
	\begin{equation}\label{eq:tbar-ti-def}
		|u(s_i)|=|u(t_i)|=:U>|u(b_i)|.
	\end{equation}
	By \eqref{eq:4.20}, \[T_2(s_i)>0.\]  
	Recall the definition of $g_2$ in \eqref{eq:def-T2-g2},  and set
	\[
	a_i:=g_2(u(t_i))=\frac{2F(u(t_i))}{u(t_i)f(u(t_i))-2F(u(t_i))},
	\qquad
	F_{a_i}(u):=F(u)-\frac{a_i}{2}\bigl[uf(u)-2F(u)\bigr].
	\]
 
	By \eqref{eq:g2-explicit-sublinear}, 
	\begin{equation}\label{eq:H_i-sublinear}
		\begin{aligned}
			-2F_{a_i}(u)
			&=
			-2F(u)+a_i\bigl(uf(u)-2F(u)\bigr)  \\
			&=
			-u^2+\frac{2}{q}|u|^q
			+
			\frac{qU^{2-q}-2}{2-q}\cdot \frac{2-q}{q}|u|^q =
			U^{2-q}|u|^q-u^2 .
		\end{aligned}
	\end{equation}
We have 
$-2F_{a_i}(u(r)) > 0$ for  $r\in(s_i, t_i)$ except $r=z_i$,   and 
	$
	F_{a_i}(u(t_i))=F_{a_i}(u(s_i))=0
	$.

	By \eqref{eq:Ba-sublinear} and \eqref{eq:Baprime-sublinear}, 
	\begin{equation}\label{eq:Bai-prime-sublinear}
		B_{a_i}(r):=
		Q(r)-a_iM(r)-2F_{a_i}(u(r))\frac{r^{n-1}v(r)}{u'(r)},\qquad
		B_{a_i}'(r)
		=
		-2F_{a_i}(u)\frac{Q_n(r)}{r u'(r)^2}.
	\end{equation}
	At \(r=s_i,t_i\), the term involving \(F_{a_i}\) vanishes, and hence by  \eqref{eq:def-T2-g2}, \eqref{eq:Bai-prime-sublinear}, and the definition of $a_i$, we have 
	\[
	B_{a_i}(s_i)=T_2(s_i),
	\qquad
	B_{a_i}(t_i)=T_2(t_i).
	\]
	Therefore, \eqref{eq:Bai-prime-sublinear} gives
	\begin{equation}\label{eq:T2-bridge-difference-sublinear}
		T_2(t_i)-T_2(s_i)=B_{a_i}(t_i)-B_{a_i}(s_i)
		=
		\int_{s_i}^{t_i}
		-2F_{a_i}(u)\frac{Q_n(r)}{r u'(r)^2}\,dr .
	\end{equation}
	If the integral in \eqref{eq:T2-bridge-difference-sublinear} is positive, then $T_2(r) \ge T_2(t_i) > T_2(s_i) > 0$ for all $r \in [\bar b_i, c_i]$, since $t_i$ is the minimum of $T_2$ on this interval. 

 For $n=2$, we have $Q_n = Q_2 > 0$ and $-F_{a_i}(u) >0$ on $(s_i, t_i)$ by \eqref{eq:4.20} and \eqref{eq:H_i-sublinear},  which implies $T_2(t_i) > T_2(s_i) > 0$. 

 Thus, it remains to show that the integral in \eqref{eq:T2-bridge-difference-sublinear} is positive for $n \ge 3$.We split
	\[
	[s_i,t_i]
	=
	[s_i,b_i]\cup[b_i,\bar b_i]\cup[\bar b_i,t_i].
	\]
	First consider the middle part \([b_i,\bar b_i]\). If \(\tau_i>b_i\), then  Lemma
	\ref{lem:residual-integral-alln-sublinear}, applied with \(\widetilde u=U\), \eqref{eq:tbar-ti-def},  and \eqref{eq:H_i-sublinear} give
	\begin{equation}\label{eq:middle-residual-positive}
		\int_{b_i}^{\tau_i}
		-2F_{a_i}(u)\frac{Q_n(r)}{r u'(r)^2}\,dr>0.
	\end{equation}
	If \(\tau_i\le b_i\), this term is absent. On the rest of \([b_i,\bar b_i]\),  namely on
	$
	[\max\{b_i,\tau_i\},\bar b_i],
	$
	$v$ has already passed its unique zero \(\tau_i\), while \(u'\)
	keeps a fixed sign throughout Phase \(i\). So Lemma  \ref{lem:locating-v-zero-sublinear} implies that \(u'v>0 \) on $
	(\max\{b_i,\tau_i\},\bar b_i]$, and therefore \(Q_n\geq Q_1>0\). Since \(|u(r)|\le\alpha_*\le U\) on
	\((b_i,\bar b_i]\), one has \(-2F_{a_i}(u(r))\ge0\). Consequently,
	\begin{equation}\label{eq:middle-nonnegative}
		\int_{b_i}^{\bar b_i}
		-2F_{a_i}(u)\frac{Q_n(r)}{r u'(r)^2}\,dr\ge0.
	\end{equation}
	
	We next treat the two outer parts. For \(\mu\in[\alpha_*,U]\), let
	$
	r_{i\mu}\in[s_i,b_i],
	\ 
	\bar r_{i\mu}\in[\bar b_i,t_i]
	$
	be determined by
	\[
	|u(r_{i\mu})|=|u(\bar r_{i\mu})|=\mu.
	\]
	By Lemma \ref{lem:comparison-sublinear},
	\begin{equation}\label{eq:phi-comparison-T2}
		\phi(\bar r_{i\mu})>\phi(r_{i\mu}),
		\qquad\text{with}\quad
		\phi(r)=\frac{Q(r)}{r^{n-1}|u'(r)|}.
	\end{equation}
	Set
	\[
	\phi_n(r):=\frac{Q_n(r)}{r|u'(r)|^3}.
	\]
	We claim that
	\begin{equation}\label{eq:psi-pair-positive}
		\phi_n(\bar r_{i\mu})+\phi_n(r_{i\mu})>0,
		\qquad \mu\in(\alpha_*,U).
	\end{equation}
	 By \(\bar r_{i\mu}\in[\bar b_i,t_i]\subset(\tau_i,c_i)\),  \(u'v>0\).
	Since \(Q>0\) on \([\bar b_i,c_i]\), we have
	\[
	Q_n(\bar r_{i\mu})
	=
	Q(\bar r_{i\mu})+n(\bar r_{i\mu})^{n-1}u'v
	>
	n(\bar r_{i\mu})^{n-1}|u'v|,
	\]
	and hence
	\begin{equation}\label{eq:psi-plus-lower}
		\phi_n(\bar r_{i\mu})
		>
		n(\bar r_{i\mu})^{n-2}
		\frac{|v(\bar r_{i\mu})|}{|u'(\bar r_{i\mu})|^2}.
	\end{equation}
	If \(\phi_n(r_{i\mu})\ge0\), then \eqref{eq:psi-pair-positive} follows. Suppose
	\(\phi_n(r_{i\mu})<0\). Since \(Q_2>0\) and
	\(Q_n=Q_2+(n-2)r^{n-1}u'v\), this implies \(u'v<0\) at \(r_{i\mu}\), and
	\[
	-Q_n(r_{i\mu})
	<
	(n-2)(r_{i\mu})^{n-1}|u'v|.
	\]
	Thus
	\begin{equation}\label{eq:psi-minus-upper}
		-\phi_n(r_{i\mu})
		<
		(n-2)(r_{i\mu})^{n-2}
		\frac{|v(r_{i\mu})|}{|u'(r_{i\mu})|^2}.
	\end{equation}
	We now estimate the \(v\)-terms. Recall the connection identity \eqref{eq:connection-identity-sublinear} and \eqref{eq:theta-sublinear},
	\begin{equation}\label{eq:connection-identity-T2-proof}
		Q(r)-P(r)\frac{v(r)}{u(r)}
		=
		\omega(r)(M(r)-\theta(r)),
		\qquad
		\theta(r):=\frac{2-q}{q}\frac{r^{n-1}v(r)}{u'(r)}|u(r)|^q .
	\end{equation}
	 	Because \(u'v<0\) at \(r_{i\mu}\), \(v\) has the unique zero \(\tau_i\) in Phase \(i\), and \(u'\) has a
	fixed sign in this phase, we must have \(r_{i\mu}<\tau_i\).
	Then the sign record gives \(uv>0\) on
	\((c_{i-1},r_{i\mu}]\). Hence
	\[
	M'(r)
	=
	r^{n-1}(uf'(u)-f(u))v
	=
	(2-q)r^{n-1}u|u|^{q-2}v>0
	\qquad\hbox{on }(c_{i-1},r_{i\mu}].
	\]
	Using the assumption \(M(c_{i-1})>0\), we obtain
	$
	M(r_{i\mu})>0.
	$ Moreover, since \(uu'<0,\ u'v<0\) $$\omega=-\frac{ru'}u>0,\quad \theta=\frac{2-q}{q}\frac{r^{n-1}v}{u'}|u|^q<0.$$ 
	  Hence \(M-\theta>0\) at $r_{i\mu}$. By \eqref{eq:connection-identity-T2-proof}, we have
	$
	Q-P\frac vu=\omega(M-\theta)>0
	$ at $r_{i\mu}$.
	Thus
	$
	P\frac vu<Q$ at $r_{i\mu}$.
	Since \(uv>0\) and \(|u(r_{i\mu})|=\mu\), we have
	\[
	\frac vu=\frac{|v|}{|u|}=\frac{|v|}{\mu}.
	\]
	By Proposition \ref{pro2.9}, $P>0$, and hence
	\begin{equation}\label{eq:v-left-bound}
		|v(r_{i\mu})|
		<
		\frac{Q(r_{i\mu})}{P(r_{i\mu})}\,\mu .
	\end{equation}
	At \(\bar r_{i\mu}\), since \(\bar r_{i\mu}\in[\bar b_i,t_i]\), we have \(uv>0\),
	\(uu'>0\), and \(\omega<0\). Since \(t_i\) is the unique critical point and   minimum of $T_2$, there holds \(T_2'<0\) on \((\bar b_i,t_i)\).
	From
	\[
	T_2'=-q\frac{uu'}{|u|^q}(M-\theta)
	\]
	we obtain \(M-\theta>0\). Therefore \eqref{eq:connection-identity-T2-proof} gives $
	P\frac vu>Q$ at $\bar r_{i\mu}$, which implies that
	\begin{equation}\label{eq:v-right-bound}
		|v(\bar r_{i\mu})|
		>
		\frac{Q(\bar r_{i\mu})}{P(\bar r_{i\mu})}\,\mu .
	\end{equation}

	Combining \eqref{eq:psi-plus-lower} with \eqref{eq:v-right-bound}, we obtain
	\begin{equation}\label{eq:psi-plus-QP-lower}
		\phi_n(\bar r_{i\mu})
		>
		n \mu\,(\bar r_{i\mu})^{n-2}
		\frac{Q(\bar r_{i\mu})}
		{P(\bar r_{i\mu})|u'(\bar r_{i\mu})|^2}.
	\end{equation}
	Similarly, combining \eqref{eq:psi-minus-upper} with \eqref{eq:v-left-bound}, we get
	\begin{equation}\label{eq:psi-minus-QP-upper}
		-\phi_n(r_{i\mu})
		<
		(n-2)\mu\,(r_{i\mu})^{n-2}
		\frac{Q(r_{i\mu})}
		{P(r_{i\mu})|u'(r_{i\mu})|^2}.
	\end{equation}
	Therefore, to prove
	\eqref{eq:psi-pair-positive}
	it is enough to show that
	\begin{equation}\label{eq:key-QP-compare}
		(\bar r_{i\mu})^{n-2}
		\frac{Q(\bar r_{i\mu})}
		{P(\bar r_{i\mu})|u'(\bar r_{i\mu})|^2}
		>
		(r_{i\mu})^{n-2}
		\frac{Q(r_{i\mu})}
		{P(r_{i\mu})|u'(r_{i\mu})|^2}.
	\end{equation}
	By the definition of \(\phi\) in \eqref{eq:phi-comparison-T2},
	\[
	Q(r)=\phi(r)r^{n-1}|u'(r)|.
	\]
	Hence
	\[
	r^{n-2}\frac{Q(r)}{P(r)|u'(r)|^2}
	=
	r^{n-2}
	\frac{\phi(r)r^{n-1}|u'(r)|}{P(r)|u'(r)|^2}
	=
	r^{n-3}\frac{\phi(r)}{|u'(r)|}\frac{r^n}{P(r)}.
	\]
	Now use the following four comparisons at the paired points \(r_{i\mu}<\bar r_{i\mu}\):
	$
	\phi(\bar r_{i\mu})>\phi(r_{i\mu})
	$
	by Lemma \ref{lem:comparison-sublinear};
	$
	|u'(\bar r_{i\mu})|<|u'(r_{i\mu})|
	$
	by Proposition \ref{pro2.5} (ii), since the two points have the same value of
	\(|u|\) and \(\bar r_{i\mu}>r_{i\mu}\);
	\[
	\frac{(\bar r_{i\mu})^n}{P(\bar r_{i\mu})}
	>
	\frac{(r_{i\mu})^n}{P(r_{i\mu})}
	\]
	because \(P(r)/r^n\) is strictly decreasing by Proposition \ref{pro2.9} (iii);
	and
	$
	(\bar r_{i\mu})^{n-3}\ge (r_{i\mu})^{n-3}.
	$ 
	Multiplying these inequalities gives \eqref{eq:key-QP-compare}, which is  \eqref{eq:psi-pair-positive}.

	Changing variables \(s=|u(r)|\) on the two outer intervals yields
	\[
	\begin{aligned}
		&\int_{s_i}^{b_i}
		-2F_{a_i}(u(r))\frac{Q_n(r)}{r u'(r)^2}\,dr
		+
		\int_{\bar b_i}^{t_i}
		-2F_{a_i}(u(r))\frac{Q_n(r)}{r u'(r)^2}\,dr  \\
		&\qquad \qquad \qquad=
		\int_{\alpha_*}^{U}
		-2F_{a_i}(s)\bigl[
		\phi_n(r_{is})+\phi_n(\bar r_{is})
		\bigr]\,ds>0.
	\end{aligned}
	\]
	Here \(-2F_{a_i}(s)=U^{2-q}s^q-s^2\ge0\) for \(s\in[\alpha_*,U]\), and it is not
	identically zero.
	
	Combining this positive contribution from the outer intervals with
	\eqref{eq:middle-nonnegative}, and using \eqref{eq:middle-residual-positive} when
	\(\tau_i>b_i\), we conclude that
	\[
	\int_{s_i}^{t_i}
	-2F_{a_i}(u(r))\frac{Q_n(r)}{r u'(r)^2}\,dr>0.
	\]
	Then by \eqref{eq:T2-bridge-difference-sublinear},
	$
	T_2(t_i)>T_2(s_i)>0.
	$
	Since \(t_i\) is the unique minimum point of \(T_2\) on \((\bar b_i,c_i)\), it
	follows that
	$
	T_2(r)>0,\  r\in[\bar b_i,c_i].
	$
	The proof is complete.
\end{proof}

\subsection{Proof of the phase transition lemma}

\begin{proof}[Proof of Lemma \ref{lem:phase-transition-sublinear}]
	By Lemma \ref{lem:3.6} and Lemma \ref{lem:T2-positive-sublinear}, we have
	\[
	Q(c_1)>0,\qquad M(c_1)>0,\qquad T_2(r)>0,\quad r\in[\bar b_1,c_1].
	\]
	which proves the first assertion.
	
	Throughout the proof, whenever the positivity of \(P\) is used, it follows from
	Proposition \ref{pro2.9} if \(n\ge3\), and from Proposition \ref{pro2.9-planar}
	if \(n=2\).
	
	We need to confirm the statements (i), (ii), and (iii) under the assumption that
	\[
	Q(c_{i-1})>0,\qquad M(c_{i-1})>0,\qquad T_2(c_{i-1})>0\qquad \text{for some }i\in\{2,\ldots,k\}.
	\]
	By Lemma \ref{lem:locating-v-zero-sublinear}, \(v\) changes sign in
	\((c_{i-1},r_i)\). Let \(\tau_i\) be the first zero of \(v\) in this interval.
	Then
	\[
	\tau_i\in(c_{i-1},r_i),
	\]
	and on \((c_{i-1},\tau_i)\) the sign record from
	Lemma \ref{lem:locating-v-zero-sublinear} gives
	\begin{equation}\label{eq:phase-proof-sign-1}
		uf(u)>0,\qquad f'(u)>0,\qquad uv>0,\qquad u'v'>0,
	\end{equation}
	and
	\begin{equation}\label{eq:phase-proof-sign-2}
		uu'<0,\qquad uv'<0,\qquad u'v<0,\qquad vv'<0.
	\end{equation}
	Consequently, \eqref{AF'} gives
	\begin{equation}\label{Q'p}
	Q'(r)=2r^{n-1}f(u)v>0,
	\qquad
	M'(r)=r^{n-1}(uf'(u)-f(u))v>0\qquad \text{on }(c_{i-1},\tau_i).
	\end{equation}
	Hence
	\begin{equation}\label{eq:phase-QM-before-tau}
		Q(r)>Q(c_{i-1})>0,\qquad
		M(r)>M(c_{i-1})>0,
		\qquad r\in(c_{i-1},\tau_i].
	\end{equation}
	
	We proceed in steps.
	
	\medskip
	\noindent{\bf Step 1. Positivity of \(T_2\) on \([c_{i-1},b_i]\).}
	
	We prove
	\begin{equation}\label{eq:T2-left-positive-phase}
		T_2(r)>0,\qquad T_2'(r)>0,
		\qquad r\in[c_{i-1},b_i].
	\end{equation}
	Recall \eqref{eq:theta-sublinear} and \eqref{eq:T2prime-second-sublinear} that
	\begin{equation}\label{eq:T2prime-phase-proof}
		T_2'(r)
		=
		-q\frac{u(r)u'(r)}{|u(r)|^q}\bigl(M(r)-\theta(r)\bigr),
		\qquad
		\theta(r):=\frac{2-q}{q}
		\frac{r^{n-1}v(r)}{u'(r)}|u(r)|^q .
	\end{equation}
	On \((c_{i-1},\tau_i]\), by the sign record we have
	\(uv>0\), \(uu'<0\), \(u'v<0\), and \(M>0\). Hence \(\theta<0\), and so
	\(M-\theta>0\). Therefore
	\begin{equation}\label{eq4.54}
		T_2'
		>0
		\quad\hbox{on }(c_{i-1},\tau_i];\qquad T_2(\tau_i)>T_2(c_{i-1})>0.
	\end{equation}

	If \(b_i\le\tau_i\), 
	\eqref{eq:T2-left-positive-phase} follows directly.
	Assume now that \(\tau_i<b_i\). We first show that
	\begin{equation}\label{eq:uv-negative-tau-bi-phase}
		uv<0,\qquad r\in(\tau_i,b_i].
	\end{equation}
	Suppose \eqref{eq:uv-negative-tau-bi-phase} is not true. Since \(uv<0\) immediately to the right of \(\tau_i\), there
	exists a first zero \(\sigma\in(\tau_i,b_i]\) of \(v\). On
	\((\tau_i,\sigma)\), we have \(uv<0\) and \(u'v>0\), hence $$\theta>0\quad\text{on }(\tau_i,\sigma).$$ At
	\(r=\tau_i\), one has \(\theta=0\) and \eqref{eq:phase-QM-before-tau}: \(M>0\). 
	  As \(u\) does not change sign on $(c_{i-1},z_i)$,
	\(v\) approaches zero at \(\sigma\) from the sign opposite to that of \(u\).
	Moreover, \(v'(\sigma)\neq0\), otherwise the linearized equation would imply
	\(v\equiv0\). Hence
	$
	u(\sigma)v'(\sigma)>0.
	$
	Therefore
	\[
	M(\sigma)
	=
	\sigma^{n-1}(u'v-uv')\big|_{r=\sigma}
	=
	-\sigma^{n-1}u(\sigma)v'(\sigma)<0.
	\]
	So $M(\sigma)-\theta(\sigma)<0$. By this and
	$M(\tau_i)-\theta(\tau_i)>0$, 
	there exists a first point \(\tilde\tau_i\in(\tau_i,\sigma)\) such that
	\[
	M(\tilde\tau_i)=\theta(\tilde\tau_i).
	\]
	Before \(\tilde\tau_i\), one has \(M-\theta>0\). Since \(uu'<0\) on
	\((\tau_i,b_i]\), \eqref{eq:T2prime-phase-proof} gives \(T_2'>0\) on
	\((\tau_i,\tilde\tau_i)\). Therefore by \eqref{eq4.54}, 
	\begin{equation}\label{t22}T_2(\tilde\tau_i)>T_2(\tau_i)>0.
	\end{equation}
	
	By the connection identity \eqref{eq:connection-identity-sublinear},
	\[
	Q(\tilde\tau_i)-P(\tilde\tau_i)\frac{v(\tilde\tau_i)}{u(\tilde\tau_i)}
	=
	\omega(\tilde\tau_i)\bigl(M(\tilde\tau_i)-\theta(\tilde\tau_i)\bigr)=0.
	\]
	Since \(uv<0\) at \(\tilde\tau_i\) and \(P>0\), we get \(Q(\tilde\tau_i)<0\).
	Moreover, \(M(\tilde\tau_i)=\theta(\tilde\tau_i)>0\). Since  \(\tilde\tau_i<\sigma\leq b_i\), then
	\(F(u(\tilde\tau_i))>0\), hence  \(g_2(u(\tilde\tau_i))=2F(u)/(uf(u)-2F(u))>0\) (see \eqref{eq:def-T2-g2}), and
	$$
	T_2(\tilde\tau_i)
	=
	Q(\tilde\tau_i)-g_2(u(\tilde\tau_i))M(\tilde\tau_i)<0,
	$$
	which contradicts \eqref{t22}.  This proves \eqref{eq:uv-negative-tau-bi-phase}.
	
	We now prove \(T_2'>0\) on \((\tau_i,b_i]\). Suppose not, and let
	\(\tau_i^*\in(\tau_i,b_i]\) be the first point such that
	\(T_2'(\tau_i^*)=0\). Then \eqref{eq4.54} gives \(T_2(\tau_i^*)>0\), and by
	\eqref{eq:T2prime-phase-proof},
	\[
	M(\tau_i^*)=\theta(\tau_i^*).
	\]
	Using \eqref{eq:uv-negative-tau-bi-phase}, the same connection-identity argument
	as above gives
	\[
	Q(\tau_i^*)<0,\qquad M(\tau_i^*)=\theta(\tau_i^*)>0.
	\]
	Since \(\tau_i^*\le b_i\), then \(g_2(u(\tau_i^*))\ge 0\), so
	\(T_2(\tau_i^*)<0\), a contradiction.  
	Therefore \(T_2'>0\) on \((\tau_i,b_i]\), and
	\eqref{eq:T2-left-positive-phase} follows.
	
	\medskip
	\noindent{\bf Step 2. Positivity of \(Q,Q_1,Q_2\) on the left part.}
	
	We prove
	\begin{equation}\label{eq:left-QQ1Q2-phase}
		Q(r)>0,\qquad Q_1(r)>0,\qquad Q_2(r)>0,
		\qquad r\in[c_{i-1},\max\{b_i,\tau_i\}].
	\end{equation}
	On \((c_{i-1},\tau_i)\), by \eqref{eq:phase-proof-sign-1},
	\eqref{eq:phase-proof-sign-2},   we have that
	\[Q'(r)=2r^{n-1}f(u)v>0,\quad
	Q_1'(r)=r^{n-1}\bigl(u'v'+f(u)v\bigr)>0,
	\quad
	Q_2'(r)=2r^{n-1}u'v'>0.
	\]
	Since \(u'(c_{i-1})=0\), the definitions of \(Q_1,Q_2\) give
	\[
	Q(c_{i-1})=Q_1(c_{i-1})=Q_2(c_{i-1})>0.
	\]
	Thus \(Q,Q_1,Q_2>0\) on \([c_{i-1},\tau_i]\).
	
	If \(b_i\le\tau_i\), then \eqref{eq:left-QQ1Q2-phase} follows. Suppose
	\(\tau_i<b_i\). By \eqref{eq:T2-left-positive-phase}, \eqref{eq:uv-negative-tau-bi-phase} in Step 1, on \((\tau_i,b_i]\) we have \(uv<0\), \(T_2'>0\),
	and \(uu'<0\). Hence, \(u'v>0\), and \eqref{eq:T2prime-phase-proof} gives 
	\(M>\theta>0\). Moreover, \((\tau_i,b_i]\), \(F(u)\geq 0\)  and \(g_2(u)\geq 0\). Consequently,
	\[
	Q_2>Q_1>Q=T_2+g_2(u)M>0
	\qquad\hbox{on }(\tau_i,b_i].
	\]
	This proves \eqref{eq:left-QQ1Q2-phase}.
	
	\medskip
	\noindent{\bf Step 3. Positivity of \(Q_1\) after \(\tau_i\).}
	
	We prove
	\begin{equation}\label{eq:Q1-after-tau-phase}
		Q_1(r)>0,\  r\in[\tau_i,\bar b_i]; \quad v \text{ has no zero in } (\tau_i,\bar b_i].
	\end{equation}
We first show that if $Q_1>0$ in  $[\tau_i,\tau)$ for $\tau\in(\tau_i,\bar b_i]$ then \(v\) has no zero
	in \((\tau_i,\tau]\). 
  Suppose not, let  
	\(\sigma\in(\tau_i,\tau]\) be the first zero of \(v\) after \(\tau_i\).
	Then \(v\) keeps a fixed sign on \((\tau_i,\sigma)\). Since \(u'\) keeps a fixed
	sign in Phase \(i\), and \(v\) changes sign at \(\tau_i\), we have
	\[
	u'v>0\qquad\hbox{on }(\tau_i,\sigma).
	\]
	As \(v\) approaches zero at \(\sigma\) from this sign, and the zero is simple
	because \(\sigma<c_i\), we have
	\[
	u'(\sigma)v'(\sigma)<0.
	\]
	By the definition of \(Q_1\),
	$$
	Q_1(\sigma)=\sigma^n u'(\sigma)v'(\sigma)<0.
	$$
	This contradicts with $Q_1>0$ in  $[\tau_i,\tau)$ and $\sigma\in(\tau_i,\tau]$.
	Therefore \(v\) has no zero in \((\tau_i,\sigma]\).
	
	Now, proving \eqref{eq:Q1-after-tau-phase} reduces to showing that  
	$Q_1(r)>0,\  r\in[\tau_i,\bar b_i]$.
	Suppose   \(\tilde q_i\in(\tau_i,\bar b_i]\) be the first zero of
	\(Q_1\). By Step 2, \(\tilde q_i>\max\{b_i,\tau_i\}\). 
	Then \(v\) changes sign exactly once in \((c_{i-1},\tilde q_i]\) at $\tau_i$. Since \(u'\) keeps a fixed sign
	throughout Phase \(i\), and since \(u'v<0\) on \((c_{i-1},\tau_i)\), we obtain
	\[
	u'v>0\qquad\hbox{on }(\tau_i,\tilde q_i].
	\]
	In particular, since \(n\ge2\),
	\[
	Q_n=Q_1+(n-1)r^{n-1}u'v>0
	\qquad\hbox{on }(\tau_i,\tilde q_i).
	\]
	
	Consider
	\[
	B_0(r)=Q(r)-2F(u(r))\frac{r^{n-1}v(r)}{u'(r)},
	\]
	with
	\[
	B_0'(r)
	=
	-2F(u(r))\frac{Q_n(r)}{r u'(r)^2}.
	\]
	If \(\tau_i\le b_i\), then \(F(u)<0\) on \((b_i,\tilde q_i)\), and therefore
	$
	B_0(\tilde q_i)>B_0(b_i)=Q(b_i)>0.
	$
	If \(b_i<\tau_i\), then \(F(u)<0\) on \((\tau_i,\tilde q_i)\), and hence
	$
	B_0(\tilde q_i)>B_0(\tau_i)=Q(\tau_i)>0.
	$
	Thus in either case
	\begin{equation*}
		B_0(\tilde q_i)>0.
	\end{equation*}
	On the other hand, \(Q_1(\tilde q_i)=0\) gives
	$
	Q(\tilde q_i)
	+\tilde q_i^{\,n-1}u'(\tilde q_i)v(\tilde q_i) =0.
	$
	Therefore
	\[
	\begin{aligned}
		B_0(\tilde q_i)
		=
		-\tilde q_i^{\,n-1}u'v
		-
		2F(u)\frac{\tilde q_i^{\,n-1}v}{u'}  
		&=
		-2\tilde q_i^{\,n-1}
		\frac{v(\tilde q_i)}{u'(\tilde q_i)}
		\left(\frac12 u'(\tilde q_i)^2+F(u(\tilde q_i))\right) \\
		&=
		-2\tilde q_i^{\,n-1}
		\frac{v(\tilde q_i)}{u'(\tilde q_i)}E(\tilde q_i)<0.
	\end{aligned}
	\]
	Here \(u'v>0\) gives \(v/u'>0\), and \(E(\tilde q_i)>0\) by
	Proposition \ref{pro2.8} (iii), 
	since $i\leq k$, \(\tilde q_i\in(c_{i-1},c_i)\), and $u$ has $k+1$ zeros before $z_u$. This contradiction proves
	\(
	Q_1(r)>0, r\in[\tau_i,\bar b_i],
	\)
	and hence $v$ has no zero in $(\tau_i,\bar b_i]$.

	\medskip
	\noindent{\bf Step 4. The comparison }
	\begin{equation}\label{eq:Q-bbar-greater-b-phase}
		Q(\bar b_i)>Q(b_i)>0.
	\end{equation}
	
	First assume that \(\tau_i\le b_i\). Then
	$
	[b_i,\bar b_i]\subset[\tau_i,\bar b_i].
	$
	By Step 3, \(v\) has no zero in \((\tau_i,\bar b_i]\). Since \(u'\) keeps a fixed
	sign throughout Phase \(i\), and \(v\) changes sign at \(\tau_i\), we have
	\[
	u'v\ge0\qquad\hbox{on }[b_i,\bar b_i].
	\]
  Hence, using \(Q_1>0\) from
	Step 3 and \(n\ge2\), we obtain
	\[
	Q_n
	=
	Q_1+(n-1)r^{n-1}u'v
	>0
	\qquad\hbox{on }[b_i,\bar b_i].
	\]
	Since \(F(u)<0\) on \((b_i,\bar b_i)\) except at the isolated zero \(z_i\), we have
	\[
	B_0'(r)
	=
	-2F(u(r))\frac{Q_n(r)}{r u'(r)^2}>0\quad\text{on }(b_i,\bar b_i)\setminus\{z_i\}.
	\]
	Therefore \(B_0\)
	is strictly increasing from \(b_i\) to \(\bar b_i\). Since
	\(F(u(b_i))=F(u(\bar b_i))=0\), we obtain by Step 2
	\[
	Q(\bar b_i)=B_0(\bar b_i)>B_0(b_i)=Q(b_i)>0.
	\]

	Now suppose that $b_i < \tau_i$. 
	Then \eqref{Q'p} implies
	\[
	Q(\tau_i)>Q(b_i)>Q(c_{i-1})>0.
	\]
	On \((\tau_i,\bar b_i)\), Step 3 gives \(Q_1>0\), and \(v\) has no zero in
	\((\tau_i,\bar b_i]\), we have \(u'v>0\). Hence, since \(n\ge2\),
	\[
	Q_n
	=
	Q_1+(n-1)r^{n-1}u'v>0
	\quad\hbox{on }(\tau_i,\bar b_i).
	\]
	By \(F(u)<0\) on \((\tau_i,\bar b_i)\),
	\[
	B_0'(r)
	=
	-2F(u(r))\frac{Q_n(r)}{r u'(r)^2}>0
	\quad\hbox{on }(\tau_i,\bar b_i).
	\]
	Since \(v(\tau_i)=0\) and \(F(u(\bar b_i))=0\), we have
	$
	B_0(\tau_i)=Q(\tau_i),
	\
	B_0(\bar b_i)=Q(\bar b_i).
	$
	Thus
	\[
	Q(\bar b_i)=B_0(\bar b_i)>B_0(\tau_i)=Q(\tau_i)>Q(b_i)>0.
	\]
	This proves \eqref{eq:Q-bbar-greater-b-phase}.

	\medskip
	\noindent{\bf Step 5. Completion of an ordinary phase.}
	
	In Step 3, we have proved that $v$ has no zero in $(\tau_i,\bar b_i]$. We next show that \(v\) has no zero in \((\bar b_i,c_i]\). Suppose, to the contrary,
	that \(\tilde\tau_i\in(\bar b_i,c_i]\) is the first zero of \(v\) after
	\(\tau_i\). Then \(u'v>0\) on \((\tau_i,\tilde\tau_i)\), and at
	\(\tilde\tau_i\) one has \(u'v'\le0\), with equality only possibly if
	\(\tilde\tau_i=c_i\). Hence
	\[
	Q(\tilde\tau_i)
	=
	\tilde\tau_i^n u'(\tilde\tau_i)v'(\tilde\tau_i)\le0.
	\]
	On the other hand, on $(\bar b_i, c_i)$, $|u| > \alpha_*$ ensures $f(u)$ shares the sign of $u$. Given that $u$ and $v$ change sign exactly once
	in $(c_{i-1}, \tilde \tau_i)$
	at $z_i$ and $\tau_i$, respectively, $u$ and $v$ have the same sign on $[\bar b_i, \tilde\tau_i)$, yielding $f(u)v > 0$.Consequently,
	\[
	Q'(r)=2r^{n-1}f(u)v>0
	\quad\hbox{on }[\bar b_i,\tilde\tau_i).
	\]
	Therefore, by
	\eqref{eq:Q-bbar-greater-b-phase},
	$
	Q(\tilde\tau_i)>Q(\bar b_i)>0,
	$
	a contradiction. Thus \(v\) has no zero after \(\tau_i\) in Phase \(i\). Hence
	\(v\) has a unique zero \(\tau_i\) in \([c_{i-1},c_i]\), and
	\[
	\tau_i\in(c_{i-1},r_i).
	\]

	We now renew the positivity of \(Q,M,T_2\) at \(c_i\). By
	\eqref{eq:Q-bbar-greater-b-phase}, \(Q(\bar b_i)>0\). Since \(v\) has no zero in
	\((\bar b_i,c_i]\), the functions \(u\) and \(v\) have the same sign on
	\([\bar b_i,c_i]\). Moreover, on \((\bar b_i,c_i)\), one has \(|u|>\alpha_*\),
	and hence \(f(u)\) has the same sign as \(u\). Therefore
	$
	Q'(r)=2r^{n-1}f(u)v>0
	\ \hbox{on }(\bar b_i,c_i).
	$
	Consequently,
	$$
	Q(r)>Q(\bar b_i)>0, r\in [\bar b_i, c_i].
	$$
	At \(z_i\), since \(u=0\) and
	\(u'v>0\), we have
	$
	M(z_i)=z_i^{\,n-1}u'(z_i)v(z_i)>0.
	$
	On \((z_i,c_i]\), the functions \(u\) and \(v\) have the same sign. Therefore
	$
	M'(r)=r^{n-1}\bigl(uf'(u)-f(u)\bigr)v=(2-q)r^{n-1}|u|^{q-2}uv>0,
	$
	and hence
	\[
	M(c_i)>M(z_i)>0.
	\]

	It remains to prove \(T_2(c_i)>0\).
	As \eqref{eq:4.18} has been verified,
	we verify the assumptions 
	\eqref{eq:4.17} and \eqref{eq:4.20}
	of
	Lemma \ref{lem:T2-positive-sublinear}.
	First, the assumptions
	$
	Q(c_{i-1})>0,\  M(c_{i-1})>0
	$
	are exactly the induction hypotheses.
	By  Step 2 and Step 3,
	$Q_1>0$ on $[c_{i-1},\bar b_i]$.
	On the remaining interval \([\bar b_i,c_i]\), since \(v\) has no zero
	after \(\tau_i\) in Phase \(i\), \(u'v>0\) on \([\bar b_i,c_i]\). Moreover,
	as proved above,
	$
	Q(r)>0,\ r\in[\bar b_i,c_i].
	$
	Therefore
	\[
	Q_1(r)=Q(r)+r^{n-1}u'(r)v(r)>0,
	\quad r\in[\bar b_i,c_i].
	\]
	Thus, \(Q_1>0\) on the whole interval \([c_{i-1},c_i]\), and assumption \eqref{eq:4.17} is verified.

	  Step 1 gives
	\(
	T_2(r)>0 
	\),$r\in[c_{i-1},b_i]$.
	We verify \(Q_2>0\) on \([c_{i-1},c_i]\). Step 2 gives \(Q_2>0\) on
	\([c_{i-1},\max\{b_i,\tau_i\}]\). Since \(v\) has no zero after \(\tau_i\), we have
	$
	u'v\ge0
	$ on $(\tau_i,c_i]$.
	Then,
	\[
	Q_2(r)=Q_1(r)+r^{n-1}u'(r)v(r)>0\quad\text{on }(\tau_i,c_i].
	\]
	Hence
	$
	Q_2(r)>0,\  r\in[c_{i-1},c_i].
	$
	Thus 	
	all assumptions of Lemma \ref{lem:T2-positive-sublinear} are now verified.
	Therefore
	\[
	T_2(r)>0,\quad r\in[\bar b_i,c_i].
	\]

	\medskip
	\noindent{\bf Step 6. The last phase.}
	
	Assume first that \(u\) is not a bound state. Then, by the assumption in
	Lemma \ref{lem:phase-transition-sublinear}, \(u\) has exactly \(k+1\) zeros before
	\(z_u=+\infty\). Hence \(z_{k+1}\) exists, and
	\(r_{k+1}\in(c_k,z_{k+1})\) is well-defined.
	Applying Steps 1--3 with \(i=k+1\), and replacing the right endpoint
	\(\bar b_i\) by \(z_{k+1}\), we obtain
	\[
	Q_1(r)>0,\quad r\in[\tau_{k+1},z_{k+1}].
	\]
	To show $v$ has no zeros in $(\tau_{k+1},z_{k+1}]$, suppose for contradiction that $\sigma\in(\tau_{k+1},z_{k+1}]$ is the first such zero. Since $u'$ has a constant sign on $(c_k,z_{k+1})$ and $v$ changes sign at $\tau_{k+1}$, we have $u'v>0$ on $(\tau_{k+1},\sigma)$. Because $v$ approaches 0 from this side and $u'(\sigma)\neq 0$, it follows that $u'(\sigma)v'(\sigma)<0$. This yields $Q_1(\sigma)=\sigma^n u'(\sigma)v'(\sigma)<0$, contradicting $Q_1>0$ on $[\tau_{k+1},z_{k+1}]$. Thus, $v$ has no zeros in $(\tau_{k+1},z_{k+1}]$. 
	Since Lemma
	\ref{lem:locating-v-zero-sublinear} gives the first zero
	$
	\tau_{k+1}\in(c_k,r_{k+1}),
	$
	assertion (ii) follows.
	
	Assume now that \(u\) is a bound state. Applying Steps 1--3 with \(i=k+1\), and
	replacing the right endpoint \(\bar b_i\) by \(z_u\), gives
	\[
	Q_1(r)>0,\quad r\in[\tau_{k+1},z_u).
	\]
	The same zero-crossing argument shows that \(v\) has no zero in
	\((\tau_{k+1},z_u)\). Therefore \(v\) has a unique zero
	$
	\tau_{k+1}\in(c_k,r_{k+1})
	$
	in \([c_k,z_u)\). By Lemma \ref{limitv}, $v$ is eventually monotone.
	If $v$ stays bounded as $r\uparrow z_u$, then by Lemma \ref{limitv} and Proposition \ref{pro2.6} (iv),
	$ |uv'|+
	|u'v'|\to 0$, $f(u)v\to 0$, $F(u)/u'\to 0$ as $r\uparrow z_u$. 
	Hence, by the definition of $B_0$, 
	 \(
	 B_0(r)\to 0,
	 \) as $r\uparrow z_u$.
	  Since \(v\) has no zero in
	\((\tau_{k+1},z_u)\) and \(u'\) has a fixed sign in the last nodal interval, while
	\(v\) changes sign at \(\tau_{k+1}\), we have
	\[
	u'v>0,\qquad r\in(\tau_{k+1},z_u).
	\]
	Together with \(Q_1>0\), this gives
	$
	Q_n=Q_1+(n-1)r^{n-1}u'v>0,
	\  r\in(\tau_{k+1},z_u).
	$
	Let \(b_{k+1}\in(c_k,r_{k+1})\) be defined by
	$
	|u(b_{k+1})|=\alpha_*
	$.
	Then \(F(u)<0\) on \((\max\{b_{k+1},\tau_{k+1}\},z_u)\).
	Therefore, $B_0'>0$ in $(\max\{b_{k+1},\tau_{k+1}\},z_u)$ and $B_0(\max\{b_{k+1},\tau_{k+1}\})<0$.
However, 
since there holds either $F=0$ or $v=0$ at $\max\{b_{k+1},\tau_{k+1}\}$, applying \eqref{eq:left-QQ1Q2-phase} with $i=k+1$,
we have 
\[
B_0(\max\{b_{k+1},\tau_{k+1}\})=  Q(\max\{b_{k+1},\tau_{k+1}\})>0.
\]
This contradiction implies	 
	\[
	\lim_{r\uparrow z_u}|v(r)|=\infty.
	\]	
\end{proof}

\section{Proof of main result}

\begin{lemma}\label{lem:zero-structure-sublinear}
	Let \(n\ge2\). Let \(u=u(\cdot,\alpha)\) be a solution of \eqref{eq2.1} with
	\(\alpha>\alpha_*\), and let \(v=\partial_\alpha u\). Then the following
	statements hold.
	
	\begin{enumerate}
		\item[(a)] If \(u\) has exactly \(k\ge1\) simple zeros
		\(0<z_1<\cdots<z_k<z_u\), then \(v\) has exactly \(k\) zeros
		\(0<\tau_1<\cdots<\tau_k<z_k\) in \([0,z_k]\). Moreover,
		$
		\tau_1\in(0,z_1),\
		\tau_i\in(z_{i-1},z_i)\ \hbox{for }2\le i\le k,
		$
		and
		$
		u'(\tau_i)v'(\tau_i)>0,\ 1\le i\le k .
		$
		
		\item[(b)] If \(u\) is a ground state, then \(v\) has exactly one zero
		\(\tau_1\in(0,z_u)\). Moreover, \(v\) is eventually monotone as
		\(r\uparrow z_u\), and
		$
		\lim_{r\uparrow z_u}v(r)=-\infty .
		$

		\item[(c)] 	If \(u\) is a bound state with exactly \(k\ge1\) simple zeros
		\(0<z_1<\cdots<z_k<z_u\), then \(v\) has exactly \(k+1\) zeros in
		\([0,z_u)\). The first \(k\) zeros satisfy the conclusion in \((a)\), and the
		last zero satisfies
		$
		\tau_{k+1}\in(c_k,r_{k+1})\subset(z_k,z_u).
		$
		Moreover, \(v\) is eventually monotone as \(r\uparrow z_u\), and
		$
		\lim_{r\uparrow z_u}|v(r)|=\infty .
		$
	\end{enumerate}
\end{lemma}
\begin{proof}
	\medskip
	\noindent{\bf Proof of (a).}
	Assume that \(u\) has exactly \(k\ge1\)  simple zeros
	$
	0<z_1<\cdots<z_k<z_u .
	$
	If \(k=1\), then   from
	Lemma \ref{lem:no-second-zero-before-z1-sublinear}, \(v\) has exactly one
	zero \(\tau_1\in(0,r_1)\subset(0,z_1)\) before \(z_1\) with
	\(v'(\tau_1)<0\). Since \(u'<0\) on \((0,z_1)\), we obtain
	$
	u'(\tau_1)v'(\tau_1)>0.
	$

	Let \(k\ge2\). By Lemma \ref{lem:3.6} and
	Lemma \ref{lem:T2-positive-sublinear}, we have
	$
	Q(c_1)>0,\  M(c_1)>0,\  T_2(c_1)>0 .
	$
	If \(u\) is a bound state, then Lemma \ref{lem:phase-transition-sublinear} (i)
	can be applied successively for \(2\le i\le k\). Thus, for each
	\(2\le i\le k\), it gives exactly one zero \(\tau_i\) of \(v\) in the
	corresponding phase, with
	\[
	\tau_i\in(c_{i-1},r_i)\subset(z_{i-1},z_i),
	\]
	and no other zero of \(v\) in that phase. Together with the
	Phase 1 zero \(\tau_1\in(0,z_1)\), this gives exactly \(k\) zeros
	$
	0<\tau_1<\cdots<\tau_k<z_k
	$
	of \(v\) in \([0,z_k]\), with
	$$
	\tau_1\in(0,z_1),\quad
	\tau_i\in(z_{i-1},z_i),\quad  2\le i\le k .
	$$
	
	If \(u\) is not a bound state, then Lemma
	\ref{lem:phase-transition-sublinear} (i) is applied successively for
	\(2\le i\le k-1\), if \(k\ge3\). For the last zero \(z_k\), Lemma
	\ref{lem:phase-transition-sublinear} (ii), applied with \(k-1\) in place of
	\(k\), gives exactly one zero
	\[
	\tau_k\in(c_{k-1},r_k)\subset(z_{k-1},z_k),
	\]
	and no other zero of \(v\) before \(z_k\). Together with the previous zeros,
	this again gives exactly \(k\) zeros
	$
	0<\tau_1<\cdots<\tau_k<z_k
	$
	of \(v\) in \([0,z_k]\), with the stated locations.
	
	The sign condition
	$
	u'(\tau_i)v'(\tau_i)>0,\  1\le i\le k,
	$
	is given by Lemma \ref{lem:locating-v-zero-sublinear}.
	This proves (a).
	
	\medskip
	\noindent{\bf Proof of (b).}
	Assume that \(u\) is a ground state. By the ground-state part of
	Lemma \ref{lem:no-second-zero-before-z1-sublinear}, \(v\) has a unique zero
	\[
	\tau_1\in(0,r_1)\subset(0,z_u).
	\]
	The same lemma gives that \(v<0\) in \((\tau_1,z_u)\), that \(v\) is eventually
	monotone as \(r\uparrow z_u\), and that
	$
	\lim_{r\uparrow z_u}v(r)=-\infty .
	$
	Therefore (b) follows.
	
	\medskip
	\noindent{\bf Proof of (c).}
	Assume that \(u\) is a bound state with exactly \(k\ge1\) zeros
	\[
	0<z_1<\cdots<z_k<z_u .
	\]
	By (a), \(v\) has exactly \(k\) zeros in \([0,z_k]\), and these zeros satisfy the
	locations stated in (a).
	
	It remains to consider the last interval \((z_k,z_u)\). If \(k=1\), then, as
	above, Lemma \ref{lem:3.6} and Lemma
	\ref{lem:T2-positive-sublinear} give
	\[
	Q(c_1)>0,\qquad M(c_1)>0,\qquad T_2(c_1)>0 .
	\]
	If \(k\ge2\), applying Lemma \ref{lem:phase-transition-sublinear} (i) successively
	through the ordinary phases gives
	\[
	Q(c_k)>0,\qquad M(c_k)>0,\qquad T_2(c_k)>0 .
	\]
	Thus in either case the hypotheses needed to enter the last semi-tail phase at
	\(c_k\) are satisfied. Lemma \ref{lem:phase-transition-sublinear} (iii) then gives
	a unique additional zero
	\[
	\tau_{k+1}\in(c_k,r_{k+1})\subset(z_k,z_u)
	\]
	of \(v\), and no further zero in \([c_k,z_u)\). Moreover, Lemma
	\ref{lem:phase-transition-sublinear} (iii) gives that \(v\) is eventually monotone
	as \(r\uparrow z_u\), and 
	$
	\lim_{r\uparrow z_u}|v(r)|=\infty .
	$
	Therefore \(v\) has exactly \(k+1\) zeros in \([0,z_u)\), and (c) follows.
	The proof is complete.
\end{proof}

\begin{proof}[Proof of Theorem \ref{thm:main}]
	Let \(n\ge2\). For \(\alpha>0\), let \(u(r,\alpha)\) be the solution of \eqref{eq2.1}, and let
	$
	\mathcal N(\alpha)
	$
	denote the number of simple zeros of \(u(\cdot,\alpha)\) in \((0,z_u(\alpha))\). By
	Proposition \ref{pro2.8} (ii), this number is finite whenever \(u\) is nodal; if
	\(u\) does not change sign, then \(\mathcal N(\alpha)=0\).

	First we record the pausing property. Suppose that \(u(\cdot,\alpha)\) is
	oscillatory behind its last simple zero. According to Proposition \ref{pro2.8} (iv), \(u\) oscillates about \(1\) or \(-1\) behind its last zero. So there exists a
	critical point \(R\) in the tail with \(0<|u(R,\alpha)|<1\), and hence
	\(E(R,\alpha)=F(u(R,\alpha))<0\).   By continuous dependence on \(\alpha\), for every
	\(\beta\) sufficiently close to \(\alpha\), the solution \(u(\cdot,\beta)\) has the
	same number of zeros as \(u(\cdot,\alpha)\) on \((0,R]\), and
	\[
	E(R,\beta)<0
	\]
	with \(u(R,\beta)\) having the same sign as \(u(R,\alpha)\). Therefore  Proposition
	\ref{pro2.7}  applies to \(u(\cdot,\beta)\) at the point \(R\), and it follows that
	\(u(\cdot,\beta)\) has no zero after \(R\).
	 Hence
	 	$\mathcal N(\beta)=\mathcal N(\alpha)$
	for all \(\beta\) sufficiently close to \(\alpha\).

	Next we prove the jumping property at ground states and bound states. Let
	\(u(\cdot,\alpha)\) be either a ground state or a   bound state with
	\(j=\mathcal N(\alpha)\) simple zeros. We claim that
	\begin{equation}\label{eq:N-jump-main-proof}
		\mathcal N(\beta)\ge j+1
		\quad\hbox{for all }\beta>\alpha\hbox{ sufficiently close to }\alpha .
	\end{equation}
	
	By Lemma \ref{lem:zero-structure-sublinear} (b), (c), the variation \(v=\partial_\alpha u\)
	has one additional zero between the last simple zero of \(u\) and $z_u$. Moreover, \(v\) is eventually
	monotone as \(r\uparrow z_u(\alpha)\), and \(|v(r,\alpha)|\to\infty\).
	We prove the claim for the case that $u>0$ in the last nodal interval, which implies $u > 0$ and $u' < 0$ near $z_u(\alpha)$. The negative case is analogous by reversing the signs of $u$ and $v$. 

By Lemma \ref{lem:locating-v-zero-sublinear}, $u'v < 0$ prior to the last zero of $v$ in the final phase. Since $u' < 0$ in this last interval and $v$ changes sign at its final zero, Lemma \ref{lem:zero-structure-sublinear}(c) yields
\[
v(r, \alpha) < 0 \quad \text{for } r \text{ sufficiently close to } z_u(\alpha), \quad \text{and} \quad v(r, \alpha) \to -\infty \quad \text{as } r\uparrow z_u(\alpha).
\]
We can choose $R < z_u(\alpha)$ sufficiently close to $z_u(\alpha)$ such that
\[ 
	a>u(r,\alpha)>0,\quad u'(r,\alpha)<0,\quad
	\partial_\alpha u(r,\alpha)=v(r,\alpha)<0,\quad \partial_\alpha u'(r,\alpha)=v'(r,\alpha)<0,
\]
 for all $r \in [R, z_u(\alpha))$, where $a > 0$ is chosen small enough so that $f'(t) < 0$ for $t\in (0, a)$.

Then for $\beta > \alpha$   sufficiently close to $\alpha$, 
\[
0<u(R,\beta)<u(R,\alpha),\quad u'(R, \beta) < u'(R,\alpha).
\]
Since the proof of \cite[Lemma 4.4]{Cortazar1996a} relies only on $f$ being decreasing on $[0, a]$, it applies here. Its conclusion (b) yields a simple zero of $u(\cdot, \beta)$ in $(R, z_u(\alpha))$, which, by Lemma \ref{lem2.4}, must occur prior to $z_u(\beta)$. Moreover, choosing \(R\) to the right of all the \(j\) simple zeros of \(u(\cdot,\alpha)\), these zeros persist in \((0,R)\) for \(\beta\) sufficiently close to \(\alpha\) by the uniform \(C^1\) continuous dependence. Hence \(u(\cdot,\beta)\) has at least the \(j\) old simple zeros before \(R\), together with the new zero in \((R,z_u(\alpha))\). This proves claim \eqref{eq:N-jump-main-proof}.

  We also need the non-decreasing property. Let
  \(u(\cdot,\alpha)\) be either a ground state or a   bound state with
  \(j=\mathcal N(\alpha)\) simple zeros. We claim that
  \begin{equation}\label{eq:non-decreasing}
  	\mathcal N(\beta)= j
  	\quad\hbox{for all }\beta<\alpha\hbox{ sufficiently close to }\alpha .
  \end{equation}
  Let
  \(z=z_u(\alpha)\), and suppose first that
  \(u(\cdot,\alpha)>0\) in the last nodal interval. The case in which the last
  nodal interval is negative follows by applying the same argument to \(-u\).
  Choose \(a>0\) so small that \(f\) is strictly decreasing on \((0,a)\) and
  \[
  F(s)<0,\qquad 0<s<a.
  \]
  By Lemma \ref{lem:zero-structure-sublinear}, and by the compact support
  asymptotics, we may choose \(R<z\), to the right of the last simple zero of
  \(u(\cdot,\alpha)\), such that
  \[
  0<u(r,\alpha)<a,\qquad u'(r,\alpha)<0,
  \qquad R\le r<z,\qquad\text{and\ }\quad
  v(R,\alpha)<0,\qquad v'(R,\alpha)<0.
  \]
  Let \(\beta<\alpha\) be sufficiently close to \(\alpha\), and set
  \[
  w(r):=u(r,\beta)-u(r,\alpha).
  \]
  Since
  $
  w(R)=(\beta-\alpha)v(R,\alpha)+o(|\beta-\alpha|),
  \ 
  w'(R)=(\beta-\alpha)v'(R,\alpha)+o(|\beta-\alpha|),
  $
  we have
  \[
  w(R)>0,\qquad w'(R)>0.
  \]
  Moreover, by taking \(\beta\) closer to \(\alpha\), we may also assume that
  \[
  0<u(R,\beta)<a,\qquad u'(R,\beta)<0.
  \]
  
 We now show that \(u(r,\beta)\) has no zero after \(R\).
 As long as \(0<u(r,\beta)<a\) and \(w>0\), we have
 \[
 0<u(r,\alpha)<u(r,\beta)<a .
 \]
 Since \(f\) is strictly decreasing on \((0,a)\), subtracting the two
 equations gives
 \[
 \bigl(r^{n-1}w'(r)\bigr)'
 =
 -r^{n-1}\bigl(f(u(r,\beta))-f(u(r,\alpha))\bigr)>0 .
 \]
 Hence \(r^{n-1}w'\) is increasing. In particular, \(w'\) stays positive,
 and therefore \(w\) cannot have a first zero before \(u(r,\beta)\) reaches the
 level \(a\).
 
 If \(u(r,\beta)\) reaches the level \(a\) before \(z\), then, since
 \(u(R,\beta)<a\) and \(u'(R,\beta)<0\), there exists
 \(c_\beta\in(R,z)\) such that
 $
 u'(c_\beta,\beta)=0,\  0<u(c_\beta,\beta)<a .
 $
 Thus
 $
 E(c_\beta,\beta)=F(u(c_\beta,\beta))<0 .
 $
 By Proposition \ref{pro2.7}, \(u(r,\beta)\) remains positive after
 \(c_\beta\), and hence no zero occurs after \(R\).
 
 It remains to consider the case where \(u(r,\beta)<a\) for all
 \(R\le r<z\). Then the preceding comparison gives
 \[
 w'(r)>0,\qquad w(r)>0,\qquad R\le r<z .
 \]
 Since \(u'(r,\alpha)\to0\) as \(r\uparrow z\), while
 \(u'(R,\beta)<0\), there exists \(c_\beta\in(R,z]\) such that
 $
 u'(c_\beta,\beta)=0 .
 $
 At this point,
 $
 0<u(c_\beta,\beta)<a,
 $
 and hence
 $
 E(c_\beta,\beta)=F(u(c_\beta,\beta))<0 .
 $
 Again Proposition \ref{pro2.7} implies that \(u_\beta\) remains positive
 after \(c_\beta\). Therefore \(u_\beta\) has no zero in \((R,\infty)\).
 
 On the fixed interval \([0,R]\), the \(j\) simple zeros of \(u_\alpha\)
 persist uniquely by the  uniform $C^1$  continuous dependence. Thus the claim \eqref{eq:non-decreasing} holds.
  
The proof of the pausing property demonstrates that   the set of parameters $\alpha$ yielding oscillatory solutions with no double zero is open. Consequently, by Proposition \ref{pro2.8} (ii) , non-decreasing property, and the jumping property, the parameters $\alpha$ corresponding to a ground state or a bound state must be isolated points in $(0, \infty)$. Since $\mathcal{N}(\alpha)$ remains constant on the open intervals between these points and undergoes strict upward jumps only at the isolated points themselves, it follows that $\mathcal{N}$ is globally non-decreasing:
\begin{equation}\label{eq:zero-curve-main-proof}
	\mathcal{N}(\beta) \ge \mathcal{N}(\alpha), \quad \text{for all } \beta > \alpha > 0.
\end{equation}

	We now prove the classification and uniqueness assertions. Let \(\alpha_0\) be an
	initial value for which \(u(\cdot,\alpha_0)\) is a ground state, and for \(k\ge1\)
	let \(\alpha_k\) be an initial value for which \(u(\cdot,\alpha_k)\) is a
	  bound state with $k$ simple zeros. 
	By the jumping property \eqref{eq:N-jump-main-proof} at the ground state \(\alpha_0\), and the monotonicity
	\eqref{eq:zero-curve-main-proof},
	\[
	\mathcal N(\alpha)\ge1,\quad \alpha>\alpha_0.
	\]
	Hence no ground state can occur for \(\alpha>\alpha_0\). If there were another
	ground state with initial value \(\gamma<\alpha_0\), then the jumping property at
	\(\gamma\), together with monotonicity, would imply
	$
	\mathcal N(\alpha_0)\ge1,
	$
	which contradicts the fact that \(u(\cdot,\alpha_0)\) is a ground state. Therefore
	the ground state is unique.
	
	Similarly, for each \(k\ge1\), the   bound state with $k$ simple zero is unique. Indeed, if
	there were two such bound states with initial values \(\gamma<\beta\), then
	the jumping property at \(\gamma\) and monotonicity would give
	$
	\mathcal N(\beta)\ge k+1,
	$
	contradicting the fact that \(u(\cdot,\beta)\) has exactly \(k\) zeros.
	
	The same argument gives the ordering. If \(m>k\) but \(\alpha_m<\alpha_k\), then
	the jumping property at \(\alpha_m\) and monotonicity would imply
	$
	\mathcal N(\alpha_k)\ge m+1>k,
	$
	a contradiction. Hence
	\[
	\alpha_0<\alpha_1<\alpha_2<\cdots .
	\]
	Moreover,
	$
	\alpha_k\to+\infty .
	$
	Otherwise the increasing sequence \(\{\alpha_k\}\) would be bounded above by some
	\(A\). Then monotonicity and the jumping property would give
	\[
	\mathcal N(A)\ge k\ \ \text{for every }k,
	\]
	 contradicting Proposition \ref{pro2.8} (ii), which says that
	every nodal solution has only finitely many sign changes.

	For the  bound state \(u(\cdot,\alpha_k)\), Proposition
	\ref{pro2.6} (iii) gives the uniqueness of the critical point between any two
	consecutive zeros and also the unique critical point behind the last zero.
	Moreover, at each such critical point \(c_i\),
	\[
	|u(c_i)|>\alpha_*.
	\]
	Proposition \ref{pro2.6} (iv) gives compact support and the boundary behavior at the support
	radius. This proves the structure asserted for the unique bound state with $k$ simple zeros.
	
	If \(\alpha<\alpha_0\) and \(\alpha\neq1\), then
	\[
	\mathcal N(\alpha)=0.
	\]
	Indeed, otherwise monotonicity would give \(\mathcal N(\alpha_0)>0\), contradicting
	that \(u(\cdot,\alpha_0)\) is a ground state. Hence \(u(\cdot,\alpha)\) is
	positive. It cannot be a ground state by the uniqueness just proved. Therefore,
	by Proposition \ref{pro2.8} (iv), it is oscillatory; equivalently, by
	Proposition \ref{pro2.7}, it oscillates about \(1\).
	
	Finally, let
	$
	\alpha\in(\alpha_k,\alpha_{k+1}),\  k\ge0.
	$
	The jumping property at \(\alpha_k\) gives
	\[
	\mathcal N(\alpha)\ge k+1.
	\]
	On the other hand, if \(\mathcal N(\alpha)\ge k+2\), then monotonicity would give
	$
	\mathcal N(\alpha_{k+1})\ge k+2,
	$
	contradicting that \(u(\cdot,\alpha_{k+1})\) is the \((k+1)\)-node bound state.
	Thus
	\[
	\mathcal N(\alpha)=k+1.
	\]
	The solution \(u(\cdot,\alpha)\) cannot be a bound state, because the unique
	\((k+1)\)-node bound state occurs at \(\alpha_{k+1}\). Therefore Proposition
	\ref{pro2.8} (iv) implies that \(u(\cdot,\alpha)\) oscillates about \(1\) or
	\(-1\) behind its last zero.
	
	This proves the classification theorem. The uniqueness theorem for bound states
	follows immediately: for each \(k\ge1\), the radial \(k\)-node bound state is the
	solution with initial value \(\alpha_k\), and it is unique up to the symmetry
	\(u\mapsto -u\). The proof is complete.
\end{proof}

\section{Proof of uniqueness for superlinear case in dimension two}
In this section, we show that \cite[Theorems 1 and 2]{Tang2025} remain valid for $n=2$, provided $p>1$.

The main tool used is \cite[Lemma 5.1 (Phase Transition Lemma)]{Tang2025}. The subsequent arguments do not depend on the condition $n \ge 3$. In the proof of this lemma, it is critical to establish the positivity of $T_2$ on $[\bar b_i, c_i]$ (see \cite[Lemma 5.6]{Tang2025}). The original proof relied on either the positivity of $Q_n$ on $[b_i,c_i]$, which was proved in the last part of \cite[Lemma 5.3]{Tang2025} for $n=3, p\in [2,5)$ and $n\geq 4$, or the positivity of an integral for $n=3, p\in(1,2)$ established in \cite[Lemma 5.4]{Tang2025}. Both of these results appear to require $n \ge 3$. Although \cite[Lemma 5.4]{Tang2025} can now be extended to $n \ge 2$ and $p>1$ via our Lemma \ref{lem:residual-integral-alln-sublinear}, we note that for $n=2$, the function $Q_n(r)$ reduces to $Q_2(r)$. Since the positivity of $Q_2(r)$ on $[c_{i-1}, c_i]$ is already assumed a priori by the hypotheses of \cite[Lemma 5.6]{Tang2025}, the proof no longer depends on these two technical lemmas. 

Finally, we note one last difference between the cases $n=2$ and $n \ge 3$, which lies in \cite[Proposition 2.5(iii)]{Tang2025}. When $n=2$, the strict inequalities $P_2>0$ and $[-P(r)/r^n]'>0$ may fail at some isolated points. However, the proof only requires the corollary that $P(r)/r^n$ is strictly decreasing, a property used exclusively in \cite[Lemma 5.6]{Tang2025}. For $n=2$, this strictly decreasing property holds naturally because $P(r)/r^2 = 2E(r)$.

We will indicate the necessary modifications, and it suffices to supply the proof of \cite[Lemma 5.6]{Tang2025} for $n=2$.

\textbf{Notations.}
The radial profile satisfies
\begin{equation}\label{eu2}
u''+\frac{1}{r}u'+f(u)=0,\quad u(0)=\alpha>0,\; u'(0)=0,
\end{equation}
with \(f(u)=-u+|u|^{p-1}u\). The variation \(v=\partial u/\partial\alpha\) solves
\begin{equation}\label{ev2}
v''+\frac{1}{r}v'+f'(u)v=0,\quad v(0)=1,\; v'(0)=0.
\end{equation}
All energy-type functions ($E$, $P$, $Q$, $M$, $T_1$, $T_2$, and $B_a$) are adopted directly from \cite{Tang2025} by setting the dimension $n=2$. For their precise definitions and derivatives, we refer the reader to Tables 1 and 2 in the appendix of that paper. The fundamental algebraic identities and derivative formulas for these quantities remain unchanged. In particular, we have
 \[
 B_a'(r) = -2F_a(u)\frac{Q_n(r)}{r u'^2} = -2F_a(u)\frac{Q_2(r)}{r u'^2}.
 \]
 Furthermore, the basic properties of radial solutions (Propositions 2.1--2.4) hold for $n=2$ without modification, as their proofs rely only on the assumption $n > 1$.

\textbf{Phase 1 and the first zero of \(v\).}
The analysis of Section~4 in the main paper is independent of the restriction \(n\ge 3\). In particular, Lemmas 4.1--4.3 apply verbatim for \(n=2\). Consequently, there exists \(\tau_1\in(0,r_1)\) such that
\[
v>0\;\text{on}\;(0,\tau_1),\quad v(\tau_1)=0,\quad v<0\;\text{on}\;(\tau_1,c_1],
\]
and
\[
Q(r)>0,\; M(r)>0,\; Q_1(r)>0,\; Q_2(r)>0\quad\text{for all }r\in(0,c_1].
\]
Thus \(Q_2(c_1)>0\), and \(Q(c_1)>0\), \(M(c_1)>0\), \(T_2(c_1)>0\) by the same arguments as for \(n\ge 3\).

\textbf{Positivity for $T_2$ on $[\bar b_i, c_i]$.}
We prove   \cite[Lemma 5.6 ]{Tang2025} for $n=2$.
\begin{lemma}[Lemma 5.6 of \cite{Tang2025} for $n=2$]
\label{lem:5.6_n2}
Assume $n=2$. For $i \in \{2, \dots, k\}$, suppose that the following conditions hold:
\begin{align*}
& Q(c_{i-1}) > 0, \quad M(c_{i-1}) > 0, \quad \text{and} \quad Q_1(r) > 0 \quad \text{for } r \in [c_{i-1}, c_i]; \\
& Q(r) > 0 \quad \text{for } r \in (c_{i-1}, b_i] \cup [\overline{b}_i, c_i], \quad \text{and} \quad Q(\overline{b}_i) > Q(b_i); \\
& T_2(r) > 0 \quad \text{on } [c_{i-1}, b_i], \quad \text{and} \quad Q_2(r) > 0 \quad \text{on } [c_{i-1}, c_i].
\end{align*}
Then $T_2(r) > 0$ on $[\overline{b}_i, c_i]$. In particular, $T_2(r) > 0$ on $[\overline{b}_1, c_1]$.
\end{lemma}

\begin{proof}
	As in \cite{Tang2025}, the phase 1 follows same as phase $i\geq 2$,
	we only consider the case $i\in\{2,\dots,k\}$.
When $n=2$,   the hypothesis $Q_2(r)>0$ on $[c_{i-1},c_i]$   directly implies
\begin{equation}\label{eq:Qnpos}
Q_n(r)=Q_2(r)>0 \qquad \text{for all } r\in[c_{i-1},c_i].
\end{equation}

Since by the initial part of \cite[Lemma 5.3]{Tang2025}, $v$ has exactly one zero $\tau_i$ in $[c_{i-1},c_i]$ with $\tau_i\in(c_{i-1}, r_i)$,
and $u'v>0$ on $(\tau_i, c_i]$.  
Using $T_2'(r)= (p-1)r^{n-1}uv-\frac {(p+1)uu'}{|u|^{p+1}}\cdot M(r)$ we obtain $T_2'(c_i) > 0$.  
Moreover, by \cite[(5.33)]{Tang2025}, any critical point $\overline t_i$ of $T_2$ in $(\overline b_i, c_i)$ satisfies
\[
T_2''(\overline t_i) = (p-1) u u' \varphi_n(\overline t_i)>0,
\]
because $uu'>0$, $u'\neq0$, and $\varphi_n(\overline t_i)=Q_n(\overline t_i)/(\overline t_i u'^2) > 0$.
Hence $T_2$ can have at most one critical point in $(\overline b_i, c_i)$, and any such point is a strict local minimum.

Since $F(u(\overline b_i))=0$, we have $T_2(\overline b_i)=Q(\overline b_i) > 0$ by \cite[(5.32)]{Tang2025}.
If $T_2$ has no critical point in $(\overline b_i, c_i)$, then $T_2$ is strictly increasing (because $T_2'(c_i)>0$), and therefore $T_2(r) \ge T_2(\overline b_i) > 0$.

Now assume $T_2$ has a (unique) local minimum at $\overline t_i \in (\overline b_i, c_i)$.
Define $t_i \in (c_{i-1}, b_i)$ as the unique point such that $u(t_i) = -u(\overline t_i)$ since $|u(c_{i-1})|>|u(c_i)|$. Consequently $|u(t_i)| = |u(\overline t_i)| > \alpha_*$ and $t_i \in (c_{i-1}, b_i)$.

Let $a_i = g_2(u(t_i)) = g_2(u(\overline t_i))>0$ as \cite[(5.35)]{Tang2025}. As in \cite[(5.36)]{Tang2025} we have
\begin{equation}\label{eq:T2B}
T_2(t_i) = B_{a_i}(t_i), \qquad T_2(\overline t_i) = B_{a_i}(\overline t_i).
\end{equation}
Using the derivative formula of $B_a$,
\[
B_{a_i}(\overline t_i) - B_{a_i}(t_i) = \int_{t_i}^{\overline t_i} \bigl[-2F_{a_i}(u(r))\bigr] \frac{Q_n(r)}{r\,u'^2(r)}\,dr.
\]
By \cite[(5.38)]{Tang2025},
\[
-2F_{a_i}(u) = u^2\left(1 - \left|\frac{u}{u(t_i)}\right|^{p-1}\right)  \ge 0, \quad \text{for } |u| \le |u(t_i)|
\]
where   is strict for $|u| < |u(t_i)|$.
On the interval $(t_i, \overline t_i)$, the function $|u(r)|$ satisfies
\[
|u(r)| < |u(t_i)| \quad \text{for all } r \in (t_i, \overline t_i),
\]
because $|u|$ strictly decreases on $(t_i, b_i]$, then falls below $\alpha_*$ on $(b_i, z_i)$,
becomes strictly smaller than $|u(t_i)|$ on $(z_i, \overline b_i)$, and finally increases
back to $|u(t_i)|$ only at $\overline t_i$.  By $Q_n(r) > 0$   on $[c_{i-1},c_i]$, we have
\[
B_{a_i}(\overline t_i) - B_{a_i}(t_i) > 0.
\]
Combined with \eqref{eq:T2B} and the fact that $T_2(t_i) > 0$,
we obtain
\[
T_2(\overline t_i) > T_2(t_i) > 0.
\]
Since $\overline t_i$ is the unique minimizer of $T_2$ on $[\overline b_i, c_i]$, this proves $T_2(r) > 0$ for all $r\in[\overline b_i, c_i]$.
\end{proof}
 
  \appendix
 \section{Appendix: Proof of Proposition \ref{lem:singular-linear}}
 
\begin{proof}
Integrate \eqref{eq:sing}  gives
\begin{equation}\label{eq:volterra}
u(r) =u_0(r)
        + \int_{s_0}^{r} \frac{d \tau}{\tau^{m}} \int_{s_0}^{\tau} t^{m} a(t) u(t)\,dt = u_0(r) + \int_{s_0}^{r} K(r,t)\,a(t)\,u(t)\,dt, \qquad r \in I,
\end{equation}
where if $m>0$
\[
u_0(r) = b + c\,s_0^{m} \int_{s_0}^{r} \tau^{-m} d\tau,\quad K(r,t) = t^{m} \int_{t}^{r} \tau^{-m} d\tau \qquad (t>0),
\]
if $m=0$
\[
u_0(r) = b + c(r-s_0),\quad K(r,t) = (r-t).
\]
In the case $m>0$, $s_0=0$, $u_0(r) \equiv b$ is constant, and   set $K(r,0)=0$.  In the case $s_0>0$, $u_0 \in C^\infty(I)$.

In each case we have 
\[\|u_0\|_{C^1(I)} \le C_0, \qquad 
|K(r,t)| \le M_0, \qquad 
|\partial_r K(r,t)|  \le  M_1,
\]
with $C_0, M_0,M_1>0$ depending on $b,c,s_0,\delta,m$.

\medskip
\noindent\textbf{(i) Existence and uniqueness.}
Define the linear operator $T: C(I) \to C(I)$ by
\[
(Tv)(r) = u_0(r) + \int_{s_0}^{r} K(r,t)\,a(t)\,v(t)\,dt.
\]
Because $a \in L^1(I)$ and $|K| \le M_0$, the integral is well defined and continuous in $r$.  
For any $v,w \in C(I)$,
\[
|Tv(r)-Tw(r)| \le M_0 \int_{\min(s_0,r)}^{\max(s_0,r)} |a(t)|\,|v(t)-w(t)|\,dt
\le M_0 \|v-w\|_\infty \int_I |a(t)|\,dt.
\]
Although the whole interval may not give a contraction, we can split $I$ into finitely many subintervals on each of which the $L^1$ norm of $a$ is smaller than $1/(2M_0)$;  then yields a unique solution $u \in C(I)$ on each subinterval and hence on the whole $I$.  
For $r \neq s_0$, differentiate \eqref{eq:volterra}:
\[
u'(r) = u_0'(r) + \int_{s_0}^{r} \partial_r K(r,t)\,a(t)\,u(t)\,dt.
\]
The integral is continuous in $r$ because $\partial_r K$ is bounded and $a \in L^1$.  When $s_0=0$, the integral is $\int_0^{r} (t/r)^m a u\,dt$, and $|u'(r)| \le \int_0^r |a u| dt \to 0$ as $r\to0^+$, so $u'(0)=0$ and $u'$ is continuous on $[0,\delta]$.  Hence $u \in C^1(I)$.  
For $r \neq s_0$, $a$ is continuous, so the integrand is continuous and we may differentiate once more, obtaining $u''(r) = - m r^{-m-1}\int_{s_0}^r t^m a u\,dt + a(r)u(r)$.  Thus $u \in C^2(I\setminus\{s_0\})$.

\medskip
\noindent\textbf{(ii) Continuous dependence.}
Let $\bar a$ satisfy the same hypotheses and $\bar u$ be the corresponding solution with the same initial data $b,c$.  Assume $\|\bar a - a\|_{L^1(I)} \le 1$; then $\|\bar a\|_{L^1} \le \|a\|_{L^1}+1$.  
From the integral equation for $\bar u$,
\[
|\bar u(r)| \le \|u_0\|_\infty + M_0 \int_{\min(s_0,r)}^{\max(s_0,r)} |\bar a(t)|\,|\bar u(t)|\,dt.
\]
Applying Gronwall’s inequality (integrating from $s_0$ in both directions) gives
\[
\|\bar u\|_\infty \le \|u_0\|_\infty \exp\bigl(M_0 \|\bar a\|_{L^1}\bigr)
\le C_0 e^{M_0(\|a\|_{L^1}+1)} =: C_u.
\]
For the derivative,
\[
|\bar u'(r)| \le \|u_0'\|_\infty + M_1 \int_{\min(s_0,r)}^{\max(s_0,r)} |\bar a(t)|\,|\bar u(t)|\,dt
\le C_0' + M_1 \|\bar a\|_{L^1} C_u \le C_u',
\]
with $C_0' = \|u_0'\|_\infty$.  Thus $\|\bar u\|_{C^1(I)} \le C_1$ for a constant $C_1$ independent of the particular $\bar a$ satisfying the smallness condition.
 
Set $w = \bar u - u$.  Subtracting the integral equations,
\[
w(r) = \int_{s_0}^{r} K(r,t)\bigl[ \bar a(t)\bar u(t) - a(t)u(t) \bigr] dt
     = \varphi(r) + \int_{s_0}^{r} K(r,t)\,a(t)\,w(t)\,dt,
\]
where
\[
\varphi(r) = \int_{s_0}^{r} K(r,t)\,\bar u(t)\,(\bar a(t)-a(t))\,dt.
\]
Then $|\varphi(r)| \le M_0 C_u \|\bar a - a\|_{L^1(I)}$.  
From the equation for $w$ we obtain
\[
|w(r)| \le M_0 C_u \|\bar a - a\|_{L^1} + M_0 \int_{\min(s_0,r)}^{\max(s_0,r)} |a(t)|\,|w(t)|\,dt.
\]
Gronwall’s inequality yields
\[
\|w\|_\infty \le M_0 C_u \|\bar a - a\|_{L^1} \, e^{M_0 \|a\|_{L^1}} =: C_2 \|\bar a - a\|_{L^1}.
\]

For the derivative,
\[
w'(r) = \int_{s_0}^{r} \partial_r K(r,t)\bigl[ \bar a \bar u - a u \bigr] dt
      = \int_{s_0}^{r} \Bigl(\frac{t}{r}\Bigr)^{\!m} \bigl[ \bar u (\bar a - a) + a w \bigr] dt.
\]
Hence
\begin{align*}
|w'(r)| &\le M_1 \int_{\min(s_0,r)}^{\max(s_0,r)} \bigl( |\bar u|\,|\bar a - a| + |a|\,|w| \bigr) dt \\
        &\le M_1 \bigl( C_u \|\bar a - a\|_{L^1} + \|a\|_{L^1} \|w\|_\infty \bigr) \\
        &\le M_1 \bigl( C_u + \|a\|_{L^1} C_2 \bigr) \|\bar a - a\|_{L^1}.
\end{align*}
Thus $\|w'\|_\infty \le C_3 \|\bar a - a\|_{L^1}$ with $C_3 = M_1(C_u + \|a\|_{L^1} C_2)$.  

Finally,
\[
\|\bar u - u\|_{C^1(I)} \le (C_2 + C_3) \|\bar a - a\|_{L^1(I)} =: C \|\bar a - a\|_{L^1(I)},
\]
where $C$ depends only on $b,c,s_0,\delta,m$ and $\|a\|_{L^1(I)}$, as claimed.
\end{proof}

\vskip .1in
\noindent{\bf Conflict of Interest.}
On behalf of all authors, the corresponding author states that there is no conflict of interest.
\vskip .1in
\noindent{\bf Data Availability.}
No datasets were generated or analysed during the current study.

\vskip .1in
 \noindent{\bf Acknowledgement.} 
The  research was supported by   NSFC-12371107, NSFC-11901582.

\end{document}